\documentclass[a4paper]{amsart}
\usepackage{hyperref}
\usepackage{txfonts, amsmath,amstext,amsthm,amscd,amsopn,verbatim,amssymb, amsfonts}
\usepackage{fullpage}
\usepackage{todonotes}

\usepackage[bbgreekl]{mathbbol}
\usepackage{enumerate}

\usepackage{float}
\restylefloat{table}

\usepackage{tikz}
\usepackage{tikz-cd}
\usetikzlibrary{matrix}
\usetikzlibrary{shapes}
\usetikzlibrary{arrows}
\usetikzlibrary{calc,3d}
\usetikzlibrary{decorations,decorations.pathmorphing,decorations.pathreplacing,decorations.markings}
\usetikzlibrary{through}

\tikzset{snakeit/.style={decorate, decoration={snake, amplitude=.2mm,segment length=1mm}}}

\tikzset{ext/.style={circle, draw,inner sep=1pt}, int/.style={circle,draw,fill,inner sep=2pt},nil/.style={inner sep=1pt}}
\tikzset{cy/.style={circle,draw,fill,inner sep=2pt},scy/.style={circle,draw,inner sep=2pt},scyx/.style={draw,cross out,inner sep=2pt},scyt/.style={draw,regular polygon,regular polygon sides=3,inner sep=0.95pt}}
\tikzset{exte/.style={circle, draw,inner sep=3pt},inte/.style={circle,draw,fill,inner sep=3pt}}
\tikzset{diagram/.style={matrix of math nodes, row sep=3em, column sep=2.5em, text height=1.5ex, text depth=0.25ex}}
\tikzset{diagram2/.style={matrix of math nodes, row sep=0.5em, column sep=0.5em, text height=1.5ex, text depth=0.25ex}}
\tikzset{rowcolsep/.style={column sep=.2cm, row sep=.1cm}}

\tikzset{
  crossed/.style={
    decoration={markings,mark=at position .5 with {\arrow{|}}},
    postaction={decorate},
    shorten >=0.4pt}}

\tikzset{every picture/.style={baseline=-.65ex} }

\newcommand{\Ed}{{
\begin{tikzpicture}[baseline=-.8ex,scale=.5]
\node[nil] (a) at (0,0) {};
\node[nil] (b) at (1,0) {};
\draw (a) edge[-latex] (b);
\end{tikzpicture}}}

\newcommand{\dE}{{
\begin{tikzpicture}[baseline=-.8ex,scale=.5]
\node[nil] (a) at (0,0) {};
\node[nil] (b) at (1,0) {};
\draw (a) edge[latex-] (b);
\end{tikzpicture}}}

\newcommand{\EdE}{{
\begin{tikzpicture}[baseline=-.65ex,scale=.5]
 \node[nil] (a) at (0,0) {};
 \node[int] (b) at (1,0) {};
 \node[nil] (c) at (2,0) {};
 \draw (a) edge[-latex] (b);
 \draw (b) edge[latex-] (c);
\end{tikzpicture}}}

\newcommand{\EEpassing}{{
\begin{tikzpicture}[baseline=-.65ex,scale=.5]
 \node[nil] (a) at (0,0) {};
 \node[int] (b) at (1,0) {};
 \node[nil] (c) at (2,0) {};
 \draw (a) edge[-latex] (b);
 \draw (b) edge[-latex] (c);
\end{tikzpicture}}}

\newcommand{\dEd}{{
\begin{tikzpicture}[baseline=-.65ex,scale=.5]
 \node[nil] (a) at (0,0) {};
 \node[int] (b) at (1,0) {};
 \node[nil] (c) at (2,0) {};
 \draw (a) edge[latex-] (b);
 \draw (b) edge[-latex] (c);
\end{tikzpicture}}}

\newcommand{\Ess}{{
\begin{tikzpicture}[baseline=-.65ex,scale=.5]
 \node[nil] (a) at (0,0) {};
 \node[nil] (c) at (1.4,0) {};
 \draw (a) edge[crossed,->] (c);
\end{tikzpicture}}}

\newcommand{\Essu}{{
\begin{tikzpicture}[baseline=-.65ex,scale=.5]
 \node[nil] (a) at (0,0) {};
 \node[nil] (c) at (1.4,0) {};
 \draw (a) edge[crossed] (c);
\end{tikzpicture}}}

\newcommand{\ET}{{
\begin{tikzpicture}[baseline=-.65ex,scale=.5]
 \node[nil] (a) at (0,0) {};
 \node[nil] (c) at (1,0) {};
 \draw (a) edge[very thick] (c);
\end{tikzpicture}}}

\theoremstyle{plain}
  \newtheorem{thm}{Theorem}
  \newtheorem{defi}[thm]{Definition}
  \newtheorem{prop}[thm]{Proposition}
  
  \newtheorem{cor}[thm]{Corollary}
  \newtheorem{conjecture}[thm]{Conjecture}
  \newtheorem{lemma}[thm]{Lemma}
\theoremstyle{definition}
  
  \newtheorem{rem}{Remark}

\newcommand{\Hom}{\mathrm{Hom}}

\newcommand{\R}{{\mathbb{R}}}
\newcommand{\K}{{\mathbb{K}}}
\newcommand{\Z}{{\mathbb{Z}}}

\newcommand{\Def}{\mathrm{Def}}

\DeclareMathOperator{\id}{id}

\newcommand{\op}{\mathcal}

\newcommand{\Lie}{\mathsf{Lie}}

\newcommand{\Ass}{\mathsf{Assoc}}
\newcommand{\Com}{\mathsf{Com}}

\newcommand{\bpm}{\begin{pmatrix}}
\newcommand{\epm}{\end{pmatrix}}

\newcommand{\sym}{{\mathbb{S}}}

\newcommand{\grac}{\mathrm{grac}}

\newcommand{\GC}{\mathsf{GC}}
\newcommand{\OG}{\mathsf{OG}}
\newcommand{\OGC}{\mathsf{OGC}}
\newcommand{\HG}{\mathsf{HG}}
\newcommand{\HGC}{\mathsf{HGC}}
\newcommand{\HOG}{\mathsf{HOG}}
\newcommand{\HOsG}{\mathsf{HO^{sk}G}}
\newcommand{\HOGC}{\mathsf{HOGC}}

\newcommand{\HHOGC}{\mathsf{HHOGC}}
\newcommand{\HGS}{\mathsf{HG}^S}
\newcommand{\HGCS}{\mathsf{HGC}^S}
\newcommand{\HOGS}{\mathsf{HOG}^S}
\newcommand{\HOsGS}{\mathsf{HO^{sk}G}^S}
\newcommand{\HOGCS}{\mathsf{HOGC}^S}
\newcommand{\HHOGS}{\mathsf{HHOG}^S}
\newcommand{\HHOGCS}{\mathsf{HHOGC}^S}
\newcommand{\HGs}{\mathsf{HG}^s}

\newcommand{\HOGs}{\mathsf{HOG}^s}

\newcommand{\G}{\mathsf{G}}

\newcommand{\mV}{\mathsf{V}}
\newcommand{\mE}{\mathsf{E}}
\newcommand{\mH}{\mathsf{H}}
\newcommand{\mB}{\mathsf{B}}
\newcommand{\mO}{\mathsf{O}}
\newcommand{\mD}{\mathsf{D}}

\newcommand{\La}{\Lambda}
\newcommand{\F}{F}

\newcommand{\bbS}{\mathbb{S}}

\DeclareMathOperator{\sgn}{sgn}

\newcommand{\MM}{{\mathcal M}}
\newcommand{\HGK}{\mathsf{HGK}}
\newcommand{\GK}{\mathsf{GK}}
\newcommand{\RGC}{\mathsf{RGC}}
\newcommand{\Q}{\mathbb{Q}}

\newcommand{\DD}{\mathbf{D}}
\newcommand{\kk}{\mathbf{k}}
\newcommand{\FT}{\mathcal{F}}
\newcommand{\bMM}{\overline{\MM}}
\newcommand{\ori}{\mathbf{o}}
\newcommand{\gstar}{\mathrm{star}}
\newcommand{\pGrav}{\mathfrak p \mathit{Grav}}
\newcommand{\CC}{\mathcal{C}}
\newcommand{\fD}{\mathcal{D}}
\newcommand{\freeM}{\mathbb{F}}

\newcommand{\AC}{\mathsf{AC}}
\newcommand{\Frob}{\mathsf{Frob}}
\newcommand{\LieB}{\mathsf{LieB}}
\newcommand{\ILieB}{\mathsf{ILieB}}
\newcommand{\IFrob}{\mathsf{IFrob}}
\newcommand{\RGra}{\mathsf{RGra}}

\begin{document}
\title{Oriented hairy graphs and moduli spaces of curves}

%
%

\author{Assar Andersson}
\address{
University of Luxembourg\\
Maison du Nombre\\
6, Avenue de la Fonte\\
L-4364 Esch-sur-Alzette\\
Luxembourg
}
\email{assar.andersson@uni.lu}

\author{Thomas Willwacher}
\address{Department of Mathematics\\ ETH Zurich\\ R\"amistrasse 101 \\ 8092 Zurich, Switzerland}
\email{thomas.willwacher@math.ethz.ch}

\author{Marko \v Zivkovi\' c}
\address{
University of Luxembourg\\
Maison du Nombre\\
6, Avenue de la Fonte\\
L-4364 Esch-sur-Alzette\\
Luxembourg
}
\email{marko.zivkovic@uni.lu}

\thanks{
}



\begin{abstract}
We discuss a graph complex formed by directed acyclic graphs with external legs.
This complex comes in particular with a map to the ribbon graph complex computing the (compactly supported) cohomology of the moduli space of points $\MM_{g,n}$, extending an earlier result of Merkulov-Willwacher.
It is furthermore quasi-isomorphic to the hairy graph complex computing the weight 0 part of the compactly supported cohomology of $\MM_{g,n}$ according to Chan-Galatius-Payne. Hence we can naturally connect the works Chan-Galatius-Payne and of Merkulov-Willwacher and the ribbon graph complex and obtain a fairly satisfying picture of how all the pieces and various graph complexes fit together, at least in weight zero.
\end{abstract}

\maketitle

\section{Introduction}
It has been shown recently by Chan, Galatius and Payne \cite{CGP1,CGP2} that the "commutative" Kontsevich graph complex computes the top weight part of the cohomology of the moduli spaces of curves. A related, albeit weaker result had been shown by Merkulov and Willwacher \cite{MW}, who constructed a map between the Kontsevich graph complex and the Kontsevich-Penner ribbon graph complex, which also computes the cohomology of the moduli space.
The purpose of the present note is to generalize the result of Merkulov-Willwacher, slightly simplify the result of Chan-Galatius-Payne and connect the two results.

To this end, one main ingredient is the ``hairy'' graph complex $\HOGCS_n$ whose elements are formal series of isomorphism classes of directed acyclic graphs with two kinds of vertices:
\begin{itemize}
\item Internal vertices are at least bivalent and have at least one outgoing edge. There are no passing vertices, i.e.\ bivalent vertices with one incoming and one outgoing edge.
\item External vertices are 1-valent sinks, i.e.\ they have one incoming edge.
These vertices are ``distinguishable'' and identified with the elements of a given finite set $S$.
\end{itemize}
An external vertex together with the edge attach to it can be considered as a ``hair'' on an internal vertex, hence the name of the complex. We also call it ``leg''.
Here are some typical such graphs:
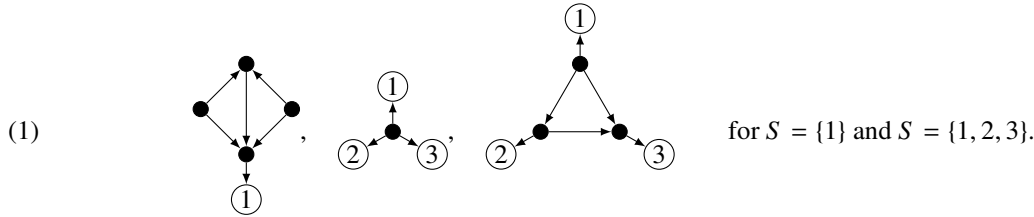
\begin{equation}\label{equ:HOGexample}
\begin{tikzpicture}[scale=.6, yshift=.5cm]
\node[int] (v1) at (0:1) {};
\node[int] (v2) at (90:1) {};
\node[int] (v3) at (180:1) {};
\node[int] (v4) at (-90:1) {};
\node[ext] (h1) at (-90:2) {1};
\draw[-latex] (v2) edge (v4) (v1) edge (v4) edge (v2) (v3) edge (v4) edge (v2) (v4) edge (h1);
\end{tikzpicture}
,\quad
\begin{tikzpicture}[scale=.6]
\node[int] (v) at (0,0){};
\node[ext] (h1) at (90:1) {1};
\node[ext] (h2) at (210:1) {2};
\node[ext] (h3) at (330:1) {3};
\draw[-latex] (v) edge (h1);
\draw[-latex] (v) edge (h2);
\draw[-latex] (v) edge (h3);
\end{tikzpicture}
,\quad
\begin{tikzpicture}[scale=.6,yshift=.5cm]
\node[int] (v1) at (90:1){};
\node[int] (v2) at (210:1){};
\node[int] (v3) at (-30:1){};
\node[ext] (h1) at (90:2) {1};
\node[ext] (h2) at (210:2) {2};
\node[ext] (h3) at (330:2) {3};
\draw[-latex] (v1) edge (h1);
\draw[-latex] (v2) edge (h2);
\draw[-latex] (v3) edge (h3);
\draw[-latex] (v1) edge (v2) edge (v3) (v2) edge (v3);
\end{tikzpicture}
\quad\quad \text{for $S=\{1\}$ and $S=\{1,2,3\}$.}
\end{equation}
The differential $\delta$ on $\HOGCS_{n}$ is given by splitting vertices. For a more precise definition of $\HOGCS_{n}$ and sign and degree conventions we refer to Section \ref{sec:HOG} below.

The graph complex $\HOGCS_{n}$ can be connected to several objects and constructions in the literature.
First, one may consider the ribbon graph complex $\RGC^{S}$ that originates from work of Penner \cite{penner} and Kontsevich \cite{K3} whose genus $g$ piece $\mB_g\RGC^S$ computes the compactly supported cohomology of the moduli space $\MM_{g,S}$ of genus $g$ curves with $|S|$ marked points labelled by the set $S$, $H^{\bullet+|S|}\left(\mB_g\RGC^{S}\right) = H_c^{\bullet}\left(\MM_{g,S}\right)$.

We show in Section \ref{sec:HOGCtoRGC} below that there is a natural map of complexes
\[
 \HOGCS_{1} \to \RGC^{S}
\]
that sends graphs of loop order $g$ to genus $g$ ribbon graphs.
We conjecture that our map induces an injection in cohomology.

On the other hand, our graph complex $\HOGCS_{1}$ is closely related to its undirected analog $\HGCS_{0}$, defined in the same manner, except that the graphs are undirected, and all internal vertices need to be at least 3-valent.
Some typical such graphs are as follows:

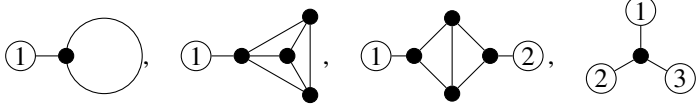
\begin{equation}\label{equ:HGexample}
 \begin{tikzpicture}[scale=.5]
\draw (0,0) circle (1);
\node[ext] (h1) at (-2.2,0) {1};
\draw (-180:1) node[int]{} edge (h1);
\end{tikzpicture}
,\quad
\begin{tikzpicture}[scale=.6]
\node[int] (v1) at (0,0){};
\node[int] (v2) at (180:1){};
\node[int] (v3) at (60:1){};
\node[int] (v4) at (-60:1){};
\node[ext] (h1) at (-2,0) {1};
\draw (v1) edge (v2) edge (v3) edge (v4) (v2) edge (v3) edge (v4) edge (h1);
\draw (v3) edge (v4);
\end{tikzpicture}
\,,\quad
\begin{tikzpicture}[scale=.5]
\node[int] (v1) at (-1,0){};
\node[int] (v2) at (0,1){};
\node[int] (v3) at (1,0){};
\node[int] (v4) at (0,-1){};
\node[ext] (h1) at (-2,0) {1};
\node[ext] (h2) at (2,0) {2};
\draw (v1)  edge (v2) edge (v4) edge (h1) (v2) edge (v4) (v3) edge (v2) edge (v4) edge (h2);
\end{tikzpicture}
\, ,\quad
\begin{tikzpicture}[scale=.6]
\node[int] (v) at (0,0){};
\node[ext] (h1) at (90:1) {1};
\node[ext] (h2) at (210:1) {2};
\node[ext] (h3) at (330:1) {3};
\draw (v) edge (h1);
\draw (v) edge (h2);
\draw (v) edge (h3);
\end{tikzpicture}
\quad\quad \text{for $S=\{1\}$, $S=\{1,2\}$ and $S=\{1,2,3\}$.}
\end{equation}
Chan, Galatius and Payne have recently shown the following result.
\begin{thm}[Chan, Galatius, Payne \cite{CGP2}]\label{thm:CGP2}
The weight 0 summand of the compactly supported cohomology of the moduli space $W_0H_c(\MM_{g,S})$ (with $2g+|S|\geq 4$) is computed by the $g$-loop part $\mB_g\HGCS_{0}$ of the graph complex $\HGCS_{0}$,
\[
H^{\bullet}\left(\mB_g\HGCS_{0}\right) \cong W_0H_c^{\bullet-|S|}\left(\MM_{g,|S|}\right).
\]
\end{thm}
We will give a simplified proof of this Theorem in Section \ref{sec:CGP} below.

The main technical result of this paper is then the following.
\begin{thm}\label{thm:main}
There is an explicit, combinatorially defined quasi-isomorphism
\[
\F : \HOGCS_{n+1} \to \HGCS_{n}
\]
which is natural in $S$ and respects the grading by loop order on both sides.
\end{thm}

We then conjecture that the following diagrams commute, possibly up to multiplying $F$ by a conventional scalar:
\[
\begin{tikzcd}
H\left(\mB_g\HGCS_{0}\right)  \ar[hookrightarrow]{r} & H_c\left(\MM_{g,S}\right)[-|S|] \ar[-]{d}{\cong} \\
H\left(\mB_g\HOGCS_{1}\right) \ar{u}{\cong}[swap]{F}  \ar{r} & H\left(\mB_g\RGC_{S}\right)
\end{tikzcd},
\]
Here the upper horizontal arrow is the map of Chan-Galatius-Payne.

Theorem \ref{thm:main} above also nicely connects with various previous results on graph complexes with hairs.
First, we may symmetrize or antisymmetrize over permutations of labels on external vertices to obtain the hairy graph complexes $\HGC_{m,n}$ which compute the rational homotopy groups of the space of long embeddings of $\R^m$ into $\R^n$, see \cite{AT, FTW}.
\[
\HGC_{m,n} = \prod_{r>0} \left( \HGC_{n}^{\{1,\dots,r\}} \otimes  \Q[-m]^{\otimes k}\right)^{\sym_r}[m]
\]

These complexes have a natural dg Lie algebra structure, realizing the Browder brackets on the embedding spaces' homotopy groups.
Furthermore, for $n=m$ there is a canonical Maurer-Cartan element
\[
m_0 = 
\begin{tikzpicture}[baseline=-.65ex]
\node[ext] (h1) at (0,0) {\;};
\node[ext] (h2) at (0.7,0) {\;};
\draw (h1) edge (h2);
\end{tikzpicture}
\in \HGC_{n,n},
\]
that corresponds to the identity map $\R^n\to \R^n$.
The corresponding twisted complex $\left(\HGC_{n,n},\delta+[m_0,-]\right)$ is essentially quasi-isomorphic to the non-hairy graph complex $\GC_n$ up to degree shift.
More precisely, the map (essentially inducing a quasi-isomorphism)
\[
\GC_n[-1] \to \left(\HGC_{n,n},\delta+[m_0,-]\right)
\]
is obtained by attaching one hair to a non-hairy graph, see \cite{TW2} for a detailed discussion.

The above construction may be paralleled for the directed acyclic hairy graph complexes $\HOGC_{n}^S$. By (anti)\-symmetrizing over hairs we obtain graph complexes $\HOGC_{m,n}$, see Section \ref{sec:HOG} for details. They also carry natural dg Lie algebra structures. Furthermore, there is a natural Maurer-Cartan element
\[
m_1 =
\sum_{k\geq 2} 
\frac 2 {k!}
\underbrace{
\begin{tikzpicture}[baseline=-.65ex]
\node[int] (v) at (0,.5) {};
\node[ext] (h1) at (-.7,-.2) {\;};
\node[ext] (h2) at (-.3,-.2) {\;};
\node[ext] (h3) at +(.7,-.2) {\;};
\draw (v) edge[-latex] (h1) edge[-latex] (h2)  edge[-latex] (h3) +(.25,-.7) node {$\scriptstyle \dots$};
\end{tikzpicture}
}_{k\times}
\in \HOGC_{n,n+1}.
\]
Twisting by this element we obtain a complex which is essentially quasi-isomorphic to the non-hairy oriented graph complex $\OGC_{n+1}$ of \cite{Woriented} (see also Proposition \ref{prop:OGCtoHOGC} below).
More precisely, one has a map of complexes
\[
\OGC_{n+1}[-1] \to (\HOGC_{n,n+1}, \delta+[m_1,-])
\]
obtained by summing over all ways of attaching hairs to a non-hairy graph, see section \ref{sec:hairynonhairOGC} below.

Now it follows from Theorem \ref{thm:main} and the construction of the map therein that $\HGC_{m,n}$ and $\HOGC_{m,n+1}$ are quasi-isomorphic as complexes. This result can be strengthened so as to cover also the dg Lie structures.
\begin{thm} \label{thm:mainlie}
	The maps of Theorem \ref{thm:main} induce a quasi-isomorphism of dg Lie algebras
	\[
	\HOGC_{m,n+1}\to \HGC_{m,n}.
	\]
	In the case $m=n$ this morphism takes the canonical MC element $m_1\in \HOGC_{n,n+1}$ to the MC element $m_0\in \HGC_{n,n}$,
and furthermore makes the following diagram of complexes commute up to homotopy in loop orders $\geq 2$.
	\begin{equation}\label{equ:thmmainliecd}
	\begin{tikzcd}
		\OGC_{n+1}[-1] \ar{r}{\simeq} \ar{d}{\simeq_{\geq 2}} & (\HOGC_{n,n+1}, \delta+[m_1,-]) \ar{d}{\simeq}\\
	\GC_n[-1] \ar{r}{\simeq_{\geq 2}} &  (\HGC_{n,n},\delta+[m_0,-]) 
	\end{tikzcd}\, .
	\end{equation}
	The left-hand vertical map is the quasi-isomorphism of \cite{MultiOriented} (up to a conventional prefactor, see section \ref{sec:GCOGCcomparison} below), and the lower horizontal map has been described in \cite{grt,TW2} (see also Proposition \ref{prop:GCHGC}). All arrows preserve the loop order.
	The arrows labelled $\simeq$ are quasi-isomorphisms, and the arrows labelled $\simeq_{\geq 2}$ are quasi-isomorphisms in loop orders $\geq 2$.
\end{thm}

Next, similar constructions for ribbon graphs have been described in \cite{MW}.
Concretely, by antisymmetrizing over the punctures one obtains the ribbon graph complex $\RGC_0$ computing the antisymmetric parts of the compactly supported cohomologies of the moduli spaces of curves.
It has been shown in \cite{MW} that $\RGC_0$ also has a dg Lie algebra structure, and there is also a canonical Maurer-Cartan element
\[
m_2 =
\begin{tikzpicture}[every loop/.style={draw,-}]
\node[int] (v) at (0,0) {};
\draw (v) edge[loop] (v);
\end{tikzpicture}
\in \RGC_0.
\]
This Maurer-Cartan element gives rise to a twisted differential on $\RGC_0$ that has first been considered by T. Bridgeland to our knowledge.
A conjecture by A. C\v ald\v araru (see Conjecture \ref{conj:caldararu} below) states that the twisted complex $( \RGC_0,\delta+[m_2,-])$ computes the compactly supported cohomology of the moduli space without marked points $\MM_g$, up to a degree shift.

Finally, there is previous work of Chan-Galatius-Payne \cite{CGP1} connecting the non-hairy graph cohomology $H(\GC_0)$ with the cohomology of the moduli space without marked points $\MM_g$. Concretely, they show the following.

\begin{thm}[Chan, Galatius, Payne \cite{CGP1}]\label{thm:CGP1}
	The weight 0 summand of the compactly supported cohomology of the moduli space $W_0H_c(\MM_{g})$ is computed by the $g$-loop part of the graph complex $\GC_{0}$,
	\[
	H\left(\mB_g\GC_{0}\right) \cong W_0H_c(\MM_{g}).
	\]
\end{thm}

The above results of Chan-Galatius-Payne and Merkulov-Willwacher together can then nicely be fit into a commutative diagram, and thus one obtains a relatively satisfying picture of all objects and morphisms involved and their relations, albeit with some conjectural components pertaining to the ribbon graphs.
For the detailed discussion we refer to Section \ref{sec:CGP} below.

\subsection*{Outline of the paper}
After some preliminary recollections in section \ref{sec:preliminaries} we discuss the definitions of various graph complexes in section \ref{sec:graph complexes}, including the new complexes $\HOGC_{n}^S$.
The proof of the main Theorem \ref{thm:main} is then given in section \ref{sec:main proof}, see in particular section \ref{sec:thedualmap} for an explicit combinatorial description of the map $F$ of Theorem \ref{thm:main}.
Section \ref{sec:HOGCtoRGC} discusses the connection to the ribbon graph complex and the work of Merkulov-Willwacher \cite{MW}.
Finally section \ref{sec:CGP} is concerned with the relation to the results of Chan-Galatius-Payne \cite{CGP1,CGP2}. In particular in section \ref{sec:big picture} we draw a picture of how everything is connected (with a conjectural component).

Let us also remark that we sometimes neglect a discussion of the genus $g\leq 1$ situation, that can also be achieved with graph complex methods, but is a somewhat special case technically that we occasionally omit in the interest of simplicity. In particular, we apologize that we partially cite results from the literature in a form (slightly) worse than what had actually been shown by the authors.

\subsection*{Acknowledgements}
The authors heartily thank Sergei Merkulov for his support and valuable discussions. We also thank A. C\v ald\v araru and A. Kalugin for their input.

T.W. has been supported by the ERC starting grant 678156 GRAPHCPX, and the NCCR SwissMAP, funded by the Swiss National Science Foundation.

\section{Preliminaries}\label{sec:preliminaries}

\subsection{Notation and conventions}
We usually work over a field $\K$ of characteristic zero, so vector spaces, algebras etc. are implicitly understood to be over $\K$.
Furthermore, the phrase "differential graded" is abbreviated \emph{dg} and often altogether omitted since most objects we study are enriched versions of differential ($\mathbb Z$-)graded vector spaces.
We generally use cohomological conventions, so that differentials have degree $+1$. For $V$ a dg vector space we denote by $V[k]$ the dg vector space in which all degrees have been shifted down by $k$ units. In other words, if $x\in V$ has degree $d$, then the corresponding element of $V[k]$ has degree $d-k$.

For brevity, we will denote by 
\[
[r] := \{1,\dots,r\}
\]
the set of numbers from $1$ to $r$.

We will use standard combinatorial terms for graphs. For example, a graph is directed if a direction is assigned to each edge and directed acyclic if one cannot inscribe a (nontrivial) directed closed path in the graph, always following the edge directions. A vertex that has only incoming edges is called a sink (or a target), and a vertex that has only outgoing edges is called a source. The number of edges incident at a vertex is the valency of the vertex.
We call an edge between a vertex and itself a tadpole. (Generally speaking, we always allow tadpoles in our graphs.)
Finally, some types of graphs may also have external legs, which we also call hairs.

\subsection{Modular operads and the Feynman transform}
We shall use the notion of modular operad from \cite{GK} which we briefly recall. A stable $\bbS$-module $M$ is a collection of right $\bbS_r$-modules $M(g,r)$ for $g,r\geq 0$ such that 
\begin{equation}\label{equ:modstability}
2g+r\geq 3.
\end{equation}
Informally, we shall think of $r$ as "the number of inputs" to some operation and $g$ as a placeholder for the genus. 
Instead of considering the collection of $M(g,r)$ we may equivalently consider a functor $S\to M(g,S)$ on the groupoid of finite sets with bijections. In other words, we label our inputs by some finite set instead of the numbers $1,\dots,r$. We shall freely switch between both conventions.

For a stable $\bbS$-module $M$ and $\Gamma$ a graph with external legs indexed by $S$ and each vertex $x$ labelled by a number $g_x\geq 0$ we can define the graph-wise tensor product 
\[
\otimes_\Gamma M = 
\otimes_{x\in V(\Gamma)} M(g_x,\gstar(x)). 
\]
Here $V(\Gamma)$ is the vertex set of $\Gamma$ and $\gstar(x)$ is the set of half-edges incident at $x$.

A modular operad is then a stable $\bbS$-module $M$ together with composition morphisms
\[
\otimes_\Gamma M \to M(g_\Gamma,S)
\]
for each graph $\Gamma$ as above with $g_\Gamma = \sum_x g_x+b_1(\Gamma)$, and $b_1(\Gamma)$ the number of loops of $\Gamma$.
The composition morphisms are required to satisfy natural coherence ("associativity") axioms. One of them is equivariance with respect to isomorphisms of graphs.

Furthermore, one can define certain twisted versions of modular operads, by twisting the aforementioned equivariance condition (essentially) by a representation of the groupoid of graphs.
For details on modular $\fD$-operads (for $\fD$ a hyperoperad) we refer to \cite[sections 4.1, 4.2]{GK}.
Dually, one obtains the notion of a modular ($\fD$-)cooperad.

To a stable $\bbS$-module $M$ we associate a corresponding free modular operad $\freeM(M)$, or more generally the free modular $\fD$-operad $\freeM_{\fD}(M)$.
In the category of dg vector spaces $\freeM(M)(g,r)$ is a dg vector space spanned by isomorphism classes of decorated graphs $\Gamma$ with $r$ legs and each vertex $x$ decorated by an element of $M(g_x,|\gstar(x)|)$, with $g=b_1(\Gamma)+ \sum_xg_x$.

The $\fD$-Feynman transform of the modular $\fD$-operad $M$ is 
\[
\FT_{\fD}(M) = ( \freeM_{\fD^\vee}(M^*), d),
\]
where $\fD^\vee=\fD^{-1}\otimes \kk$ and $\kk$ is the hyperoperad given by the top exterior power of the vector space generated by the set of edges of graphs. (Think of each edge carrying an additional cohomological degree +1 in the definition of the free modular operad.)
The differential produces precisely one additional edge, using a modular cocomposition of $M^*$ on one vertex of our graph. For more details we refer to \cite{GK}.




\subsection{PROPs and properads}
We will use the language of properads, and in particular deformation complexes associated to properad maps. For an introduction we refer the reader to \cite{MV-properad}.
We shall denote by $\LieB$ the properad governing Lie bialgebras.
A Lie bialgebra structure on a vector space $V$ consists of a Lie algebra structure and a Lie coalgebra structure (both of degree 0) satisfying a natural compatibility relation, the Drinfeld five-term identity.
In the case that the composition of the cobracket and bracket 
\[
V\xrightarrow{\text{cobracket}} V\otimes V \xrightarrow{\text{bracket}}  V
\]
is zero, then the Lie bialgebra is called involutive. We denote the corresponding properad by $\ILieB$. It comes with a natural map $\LieB\to \ILieB$.
Both of these properads are Koszul, and one can consider the resolutions $\LieB_\infty\xrightarrow{\simeq} \LieB$ and $\ILieB_\infty\xrightarrow{\simeq}\LieB$, which are obtained as the properadic cobar constructions of the Koszul dual coproperads.

To simplify signs, we shall also consider a graded version, $\La\LieB$, for which the bracket and cobracket both have cohomological degree $+1$, and are symmetric operations.
More precisely, a $\La\LieB$-structure on the graded vector space $V$ consists of a Lie algebra structure on $V[1]$, and a Lie coalgebra structure on $V[-1]$, which satisfy a graded version of the Drinfeld five-term identity.
Similarly, we may consider the graded versions of the properad of involutive Lie bialgebras $\La\ILieB$ and the corresponding resolutions $\La\LieB_\infty$, $\La\ILieB_\infty$.
For more details and recollections on these definitions we refer the reader to \cite[section 2]{MW}. (There the notation $\LieB_{0,0}$ is used in place of $\La\LieB$.)

Finally, we use the Frobenius properad $\Frob$ and its involutive version $\IFrob$. They are defined such that 
\begin{align*}
\Frob(r,s) = \IFrob(r,s)=\K
\end{align*}
for all $r\geq 1$, $s\geq 1$. All composition morphisms of $\Frob$ are the identity map. In $\IFrob$ all genus zero compositions are the identity map, while the higher genus compositions are defined to be zero.
We will use that $\La\LieB_\infty=\Omega(\IFrob^*)$, where $\Omega$ denotes the properadic cobar construction.

\section{Graph complexes}\label{sec:graph complexes}

In this section we recall the definition of various graph complexes. For those complexes that have already appeared elsewhere in the literature we just sketch the construction and provide some references. Graph complexes usually can be defined in (at least) two ways: Either completely combinatorially "by hand" or by algebraic constructions such as (pr)operadic deformation complexes.
Both ways have their advantages and disadvantages, and we shall provide or sketch both if available.

\subsection{Complex of non-hairy undirected graphs}

Let us quickly recall the combinatorial description of the simplest graph complex $\GC_n$. Consider the set $\bar\mV_v\bar\mE_e\grac^{3}$ containing directed graphs that:
\begin{itemize}
\item are connected;
\item have $v>0$ distinguishable vertices that are at least 3-valent;
\item have $e\geq 0$ distinguishable directed edges;
\end{itemize}

For $n\in\Z$, let
\begin{equation}
\bar\mV_v\bar\mE_e\G_n:=\left\langle\bar\mV_v\bar\mE_e\grac^{3}\right\rangle[(1-n)e+nv-n]
\end{equation}
be the vector space of formal linear combinations of $\bar\mV_v\bar\mE_e\bar\grac^{3}$ with coefficients in $\K$. It is a graded vector space with a non-zero term only in degree $d=(n-1)e-nv+n$.

There is a natural right action of the group $\sym_v\times \left(\sym_e\ltimes \sym_2^{\times e}\right)$ on $\bar\mV_v\bar\mE_e\bar\grac^{3}$, where $\sym_v$ permutes vertices, $\sym_e$ permutes edges and $\sym_2^{\times e}$ changes the direction of edges.
Let $\sgn_v$, $\sgn_e$ and $\sgn_2$ be one-dimensional representations of $\sym_v$, respectively $\sym_e$, respectively $\sym_2$, where the odd permutation reverses the sign. They can be considered as representations of the whole product $\sym_v\times \left(\sym_e\ltimes \sym_2^{\times e}\right)$. Let us consider the space of invariants:
\begin{equation}
\mV_v\mE_e\G_{n}:=\left\{
\begin{array}{ll}
\left(\bar\mV_v\bar\mE_e\G_n\otimes\sgn_e\right)^{\sym_v\times \left(\sym_e\ltimes \sym_2^{\times e}\right)}
\qquad&\text{for $n$ even,}\\
\left(\bar\mV_v\bar\mE_e\G_n\otimes\sgn_v\otimes\sgn_2^{\otimes e}\right)^{\sym_v\times \left(\sym_e\ltimes \sym_2^{\times e}\right)}
\qquad&\text{for $n$ odd.}
\end{array}
\right.
\end{equation}

Because the group is finite, the space of invariants may be replaced by the space of coinvariants. In any case the operation of taking (co)invarints effectively removes the edges directions and numberings of vertices and edges (up to sign).
The underlying vector space of the graph complex is
\begin{equation}
\G_{n}:=\bigoplus_{v\geq 1,e\geq 0}\mV_v\mE_e\G_{n}.
\end{equation}

The differential acts by edge contraction:
\begin{equation}
d(\Gamma)=\sum_{a\in E(\Gamma)}\Gamma/a
\end{equation}
where $E(\Gamma)$ is the set of edges of $\Gamma$ and $\Gamma/a$ is the graph produced from $\Gamma$ by contracting edge $a$ and merging its end vertices.

We may also define the dual complex 
\begin{equation}
\left(\GC_{n},\delta\right)=\left(\G_{n},d\right)^*.
\end{equation}
Here the differential $\delta$ acts combinatorially by splitting a vertex, which is the operation dual to edge contraction.

\subsection{Complexes of undirected hairy graphs}\label{sec:undirected_complexes}
\subsubsection{Description through modular operads and Feynman transform}
\label{sec:undirected_complexes_modular}
We consider the $\fD_n:=\kk^{-n}$-modular operad $\Com_n$ such that
\[
\Com_n(g,r) = 
\begin{cases}
(\sgn_S)^{\otimes n}[-rn+n] & \text{for $g=0$ and $r\geq 3$} \\
0 & \text{otherwise}
\end{cases},
\]
where $\sgn_S$ is the sign representation of the group of bijections of $S$. Then we define the graph complex 
\[
\HGCS_{n} = \prod_g \FT_{\fD_n}(\Com_n)(g,S)\otimes \sgn_S[-|S|].
\]
Combinatorially the elements of ths complex are series of graphs with $|S|$ external legs (hairs) indexed by the elements of $S$.
In the case of $S=\emptyset$ being the empty set one recovers the non-hairy graph complex of the previous subsection
\[
\GC_{n} \cong \HGC^\emptyset_n[n].
\]

For $n$ even we also consider $\fD_n$-modular operads $\Com_n^{mod}$ such that 
\[
\Com_n^{mod}(g,r) = 
\begin{cases}
(\sgn_S)^{\otimes n}[-rn+n-ng] & \text{for $2g+r\geq 3$} \\
0 & \text{otherwise}
\end{cases},
\]
and define
\[
{\HGC_n}^{S,mod} = \prod_g \FT(\Com_n)(g,S)\otimes \sgn_S[-|S|].
\]
Elements are now series of graphs with additional decorations of a number $g_x$ on each vertex $x$.

We shall only need the complex $\HGC^{S,mod}_{n}$ in the case of even $n$, and in fact $n=0$, in this paper.

\begin{lemma}[\cite{CGP2}]\label{lem:HGCmod}
For even $n$ the inclusion $\HGCS_{n}\to \HGC^{S,mod}_{n}$ coming from the modular operad map $\Com_n^{mod}\to \Com_n$ is a quasi-isomorphism in genera $\geq 2$, where the genus of a graph is the loop order plus the sum of the decorations $g_x$ on vertices. In genus one the inclusion is a quasi-isomorphism if $|S|\geq 2$.
\end{lemma}
\begin{proof}
One may filter $\HGCS_{n}$ and $\HGC^{S,mod}_{n}$ by the total number of vertices in graphs and consider the corresponding spectral sequence.
For $\HGCS_{n}$ the differential always creates exactly one vertex, and hence the differential on the associated graded is zero.
For ${\HGC_n}^{S,mod}$ there remains the piece of the differential that introduces a tadpole at a vertex $x$, and simultaneously reduces the decoration $g_x$ by one. One can easily check that the cohomology is identified with graphs with no tadpoles and all decorations $g_x=0$.
Hence on the $E^1$-pages the map $\HGCS_{n}\to {\HGCS_n}^{mod}$ becomes the map from the graph complex with tadpoles to that without tadpoles, sending all tadpoled graphs to zero.
It is known that this is essentially a quasi-isomorphism. More precisely, if $|S|>1$ the mapping cone is acyclic. For $|S|=1$, say $S=[1]$, the mapping cone has one-dimensional cohomology, spanned by the graph 
\begin{equation}\label{equ:tpwithleg}
\begin{tikzpicture}[every loop/.style={draw, -}]
\node[int] (v) at (0,0) {};
\node[ext] (w) at (0,-.6) {$\scriptstyle 1$};
\draw (v) edge[in=45,out=135,loop] (v) (v) edge (w);
\end{tikzpicture}
\in \HGC_{n}^{[1]}
\end{equation}
in genus 1. We refer to the proof of \cite[Proposition 3.4]{grt} for the detailed argument.
\end{proof}
One also has a map of complexes in the other direction
\[
\pi: \HGC^{S,mod}_{n}\to \HGCS_{n}
\]
as follows:
\begin{itemize}
	\item A graph $\Gamma\in \HGC^{S,mod}_{n}$ is sent to zero if there is a vertex $x$ with $g_x\geq 1$ and $|\gstar(x)|\geq 2$, or with $g_x\geq 2$.
	\item If our graph is of the form \eqref{equ:tpwithleg}, with any genus $g_x$ at the vertex, then we send it to zero.
	\item A vertex $x$ decorated by $g_x=1$ and of valence $1$ is sent to a tadpole at the adjacent vertex.
	
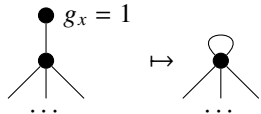
\[
	\begin{tikzpicture}[every loop/.style={draw, -},yshift=.6cm]
		\node[int,label=0:{$g_x=1$}] (v) at (0,0) {};
		\node[int] (w) at (0,-.6) {};
		\draw (v) edge (w) (w) edge +(-.5,-.5) edge +(0,-.5) edge +(.5,-.5);
		\node at (0,-1.3){$\cdots$};
	\end{tikzpicture}
	\mapsto 
	\begin{tikzpicture}[yshift=.6cm,every loop/.style={draw, -}]
		\node[int] (w) at (0,-.6) {};
		\draw (w) edge[in=45,out=135,loop] (w) (w) edge +(-.5,-.5) edge +(0,-.5) edge +(.5,-.5);
		\node at (0,-1.3){$\cdots$};
	\end{tikzpicture}
	\]
	If no such adjacent vertex exists, the graph is sent to zero by the previous rule.
	\item Otherwise the graph is sent to itself.
\end{itemize}
\begin{lemma}
The map $\pi: \HGC^{S,mod}_{n}\to \HGCS_{n}$ above (with $n$ even) is a well defined map of complexes and a one-sided inverse to that of Lemma \ref{lem:HGCmod} in genera $\geq 2$.
\end{lemma}
\begin{proof}[Proof sketch]
It is clear that our map $\pi$ is a one-sided inverse to that of Lemma \ref{lem:HGCmod}, apart from the fact that the image of the special graph \eqref{equ:tpwithleg} is sent to zero.

It hence suffices to check that $\pi$ commutes with the differentials, i.e., $\pi(\delta\Gamma)= \delta \pi(\Gamma)$ for all graphs $\Gamma \in \HGC^{S,mod}_{n}$. Suppose first that $\Gamma$ has a vertex $x$ with genus decoration $g_x\geq 2$, or with genus decoration $g_x=1$ and a tadpole at $x$. Then $\delta\Gamma$ is a linear combination of graphs with the same feature and hence $\geq 2$, and hence $\pi(\delta\Gamma)= \delta \pi(\Gamma)=0$.
Hence we can assume that each vertex $x$ of $\Gamma$ has either $g_x=0$ or $g_x=1$, and in the latter case has no tadpole at $x$.
By the same argument, we also see that $\pi(\delta\Gamma)= \delta \pi(\Gamma)=0$ if $\Gamma$ has (at least) two distict vertices $x,y$ with $g_x=g_y=1$, and both having valency $|\gstar(x)|,|\gstar(y)|\geq 2$.
If $\Gamma$ has a single vertex $x$ such that $g_x=1$ and $|\gstar(x)|\geq 2$, then $\pi(\Gamma)=0$, so we need to check $\pi(\delta \Gamma)=0$.
But the only terms produced by $\delta \Gamma$ which are potentially not send to zero are schematically of the form
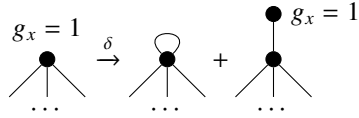
\[
\begin{tikzpicture}[every loop/.style={draw, -},yshift=.6cm]
	\node[int,label=90:{$g_x=1$}] (w) at (0,-.6) {};
	\draw (w) edge +(-.5,-.5) edge +(0,-.5) edge +(.5,-.5);
	\node at (0,-1.3){$\cdots$};
\end{tikzpicture}
\xrightarrow{\delta}
\begin{tikzpicture}[yshift=.6cm,every loop/.style={draw, -}]
	\node[int] (w) at (0,-.6) {};
	\draw (w) edge[in=45,out=135,loop] (w) (w) edge +(-.5,-.5) edge +(0,-.5) edge +(.5,-.5);
	\node at (0,-1.3){$\cdots$};
\end{tikzpicture}
+
\begin{tikzpicture}[every loop/.style={draw, -},yshift=.6cm]
	\node[int,label=0:{$g_x=1$}] (v) at (0,0) {};
	\node[int] (w) at (0,-.6) {};
	\draw (v) edge (w) (w) edge +(-.5,-.5) edge +(0,-.5) edge +(.5,-.5);
	\node at (0,-1.3){$\cdots$};
\end{tikzpicture}
\]
and cancel each other when mapped via $\pi$.
Hence we can assume that the only vertices $x$ of $\Gamma$ that have $g_x=1$ also have valency 1, and are hence of the form of the vertex on the upper right of the above picture. 
For such graphs the desired compatibility with the differentials follows from the following schematic graphical computation
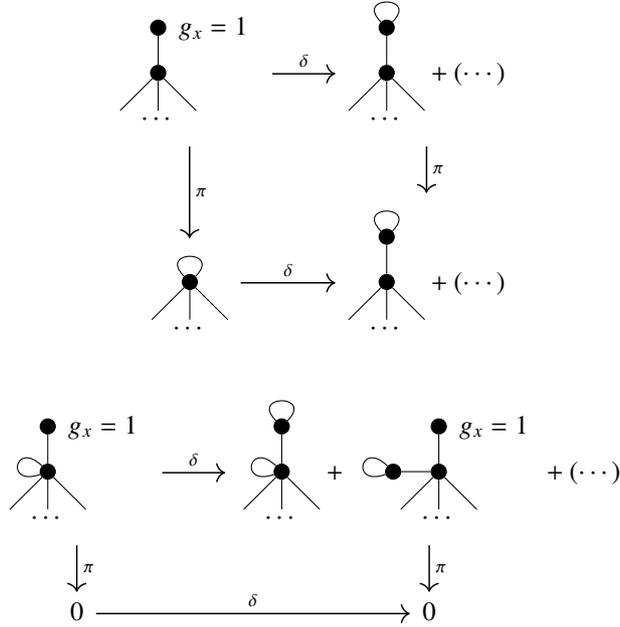
\[
\begin{tikzcd}
	\begin{tikzpicture}[every loop/.style={draw, -},yshift=.6cm]
		\node[int,label=0:{$g_x=1$}] (v) at (0,0) {};
		\node[int] (w) at (0,-.6) {};
		\draw (v) edge (w) (w) edge +(-.5,-.5) edge +(0,-.5) edge +(.5,-.5);
		\node at (0,-1.3){$\cdots$};
	\end{tikzpicture}
	\ar{r}{\delta} \ar{d}{\pi}
	&
	\begin{tikzpicture}[every loop/.style={draw, -},yshift=.6cm]
		\node[int] (v) at (0,0) {};
		\node[int] (w) at (0,-.6) {};
		\draw (v) edge[in=45,out=135,loop] (v)
		(v) edge (w) (w) edge +(-.5,-.5) edge +(0,-.5) edge +(.5,-.5);
		\node at (0,-1.3){$\cdots$};
	\end{tikzpicture}
	+(\cdots)
	\ar{d}{\pi}
	\\
	\begin{tikzpicture}[yshift=.6cm,every loop/.style={draw, -}]
		\node[int] (w) at (0,-.6) {};
		\draw (w) edge[in=45,out=135,loop] (w) (w) edge +(-.5,-.5) edge +(0,-.5) edge +(.5,-.5);
		\node at (0,-1.3){$\cdots$};
	\end{tikzpicture} 
	\ar{r}{\delta}
	&
	\begin{tikzpicture}[every loop/.style={draw, -},yshift=.6cm]
		\node[int] (v) at (0,0) {};
		\node[int] (w) at (0,-.6) {};
		\draw 
		(v) edge[in=45,out=135,loop] (v)
		(v) edge (w) (w) edge +(-.5,-.5) edge +(0,-.5) edge +(.5,-.5);
		\node at (0,-1.3){$\cdots$};
	\end{tikzpicture}
	+(\cdots)
\end{tikzcd}
\]
\[
\begin{tikzcd}
	\begin{tikzpicture}[every loop/.style={draw, -},yshift=.6cm]
		\node[int,label=0:{$g_x=1$}] (v) at (0,0) {};
		\node[int] (w) at (0,-.6) {};
		\draw (w) edge[in=200,out=135,loop] (w)
		(v) edge (w) (w) edge +(-.5,-.5) edge +(0,-.5) edge +(.5,-.5);
		\node at (0,-1.3){$\cdots$};
	\end{tikzpicture}
	\ar{r}{\delta} \ar{d}{\pi}
	&
	\begin{tikzpicture}[every loop/.style={draw, -},yshift=.6cm]
		\node[int] (v) at (0,0) {};
		\node[int] (w) at (0,-.6) {};
		\draw (w) edge[in=200,out=135,loop] (w)
		(v) edge[in=45,out=135,loop] (v)
		(v) edge (w) (w) edge +(-.5,-.5) edge +(0,-.5) edge +(.5,-.5);
		\node at (0,-1.3){$\cdots$};
	\end{tikzpicture}
	+
	\begin{tikzpicture}[every loop/.style={draw, -},yshift=.6cm]
		\node[int,label=0:{$g_x=1$}] (v) at (0,0) {};
		\node[int] (v2) at (-.6,-.6) {};
		\node[int] (w) at (0,-.6) {};
		\draw (v2) edge[in=200,out=135,loop] (v2)
		(v2) edge (w)
		(v) edge (w) (w) edge +(-.5,-.5) edge +(0,-.5) edge +(.5,-.5);
		\node at (0,-1.3){$\cdots$};
	\end{tikzpicture}
	+(\cdots)
	\ar{d}{\pi}
	\\
0
	\ar{r}{\delta}
	&
0
\end{tikzcd}
\]

\end{proof}

\begin{rem}
	The overall degree shifts in the definition of the above graph complexes are purely conventional. They indicate that we like to think of the edges as carrying degree $n-1$ and we think of the hairs as edges as well.
	This may be slightly unnatural from the  Feynman transform and moduli space perspective, but it is natural from other contexts (the embedding calculus), and it will streamline certain combinatorial constructions in the next section.
\end{rem}

\subsubsection{Combinatorial description}
\label{sss:DefUndirected}

Let us quickly recall the combinatorial description of $\HGCS_n$. It is similar to the combinatorial description of non-hairy graphs.
Consider the set $\bar\mV_v\bar\mE_e\bar\mH_S\grac^{3}$ containing directed graphs that:
\begin{itemize}
\item are connected;
\item have $v>0$ distinguishable internal vertices that are at least 3-valent;
\item have $e\geq 0$ distinguishable directed edges;
\item have $|S|\geq 0$ distinguishable 1-valent external vertices labelled by the set $S$.
\end{itemize}
For some pictures of such graphs see \eqref{equ:HGexample}.

For $n\in\Z$, let
\begin{equation}
\bar\mV_v\bar\mE_e\bar\mH_S\G_n:=\left\langle\bar\mV_v\bar\mE_e\bar\mH_S\grac^{3}\right\rangle[(1-n)e+nv]
\end{equation}
be the vector space of degree shifted formal linear combinations.

Consider again the action of the group $\sym_v\times \left(\sym_e\ltimes \sym_2^{\times e}\right)$ on $\bar\mV_v\bar\mE_e\bar\mH_S\grac^{3}$. Let
\begin{equation}
\label{eq:VEHG}
\mV_v\mE_e\bar\mH_S\G_{n}:=\left\{
\begin{array}{ll}
\left(\bar\mV_v\bar\mE_e\bar\mH_S\G_n\otimes\sgn_e\right)^{\sym_v\times \left(\sym_e\ltimes \sym_2^{\times e}\right)}
\qquad&\text{for $n$ even,}\\
\left(\bar\mV_v\bar\mE_e\bar\mH_S\G_n\otimes\sgn_v\otimes\sgn_2^{\otimes e}\right)^{\sym_v\times \left(\sym_e\ltimes \sym_2^{\times e}\right)}
\qquad&\text{for $n$ odd.}
\end{array}
\right.
\end{equation}

The underlying vector space of the hairy graph complex is
\begin{equation}
\HGS_{n}:=\bigoplus_{v\geq 1,e\geq 0}\mV_v\mE_e\bar\mH_S\G_{n}.
\end{equation}

The differential again acts by edge contraction:
\begin{equation}
d(\Gamma)=\sum_{a\in E(\Gamma)}\Gamma/a,
\end{equation}
but here $E(\Gamma)$ is the set of edges of $\Gamma$ that are not connected to an external vertex, i.e.\ edges towards an external vertex can not be contracted.

The dual complex is
\begin{equation}
\left(\HGCS_{n},\delta\right)=\left(\HGS_{n},d\right)^*,
\end{equation}
where the differential $\delta$ acts combinatorially by splitting an internal vertex.

\subsubsection{Complexes with (anti-)symmetrized hairs}

We may (anti-)symmetrize external vertices, to make them (up to the sign) indistinguishable.

Let $S=\{1,\dots,s\}$ and let $\sym_s$ (the group of bijections of $S$) act on $\HGS_n$ by permuting external vertices. Let $\sgn_s$ be one-dimensional representations of $\sym_s$, where the odd permutation reverses the sign. For an integer $m$ let
\begin{equation}
\HGs_{m,n}:=\left\{
\begin{array}{ll}
\left(\HGS_n\otimes\Q[-m]^{\otimes k}\right)^{\sym_s}[m]
\qquad&\text{for $m$ even,}\\
\left(\HGS_n\otimes\Q[-m]^{\otimes k}\otimes\sgn_s\right)^{\sym_s}[m]
\qquad&\text{for $m$ odd.}
\end{array}
\right.
\end{equation}
\begin{equation}
\HG_{m,n} = \bigoplus_{s>0} \HGs_{m,n}
\end{equation}
The differential $d$ is still contracting an edge. Again we also consider the dual complex
\begin{equation}
\left(\HGC_{m,n},\delta\right)=\left(\HG_{m,n},d\right)^*.
\end{equation}

There is a dg Lie algebra structure on $\HGC_{m,n}$ defined as
\[
\left[ 
\begin{tikzpicture}[baseline=-.8ex]
\node[draw,circle] (v) at (0,.3) {$\Gamma$};
\node[ext] (e1) at (-.5,-.4) {\;};
\node[ext] (e2) at (-.2,-.4) {\;};
\node at (.18,-.4) {$\scriptstyle\dots$};
\node[ext] (e3) at (.5,-.4) {\;};
\draw (v) edge (e1) edge (e2) edge (e3);
\end{tikzpicture}
,
\begin{tikzpicture}[baseline=-.65ex]
\node[draw,circle] (v) at (0,.3) {$\Gamma'$};
\node[ext] (e1) at (-.5,-.4) {\;};
\node[ext] (e2) at (-.2,-.4) {\;};
\node at (.18,-.4) {$\scriptstyle\dots$};
\node[ext] (e3) at (.5,-.4) {\;};
\draw (v) edge (e1) edge (e2) edge (e3);
\end{tikzpicture}
\right]
=
\sum
\begin{tikzpicture}[baseline=-.8ex]
\node[draw,circle] (v) at (0,1) {$\Gamma$};
\node[ext] (ee1) at (-.5,.3) {\;};
\node[ext] (ee2) at (-.2,.3) {\;};
\node at (.18,.3) {$\scriptstyle\dots$};
\node[draw,circle] (w) at (.8,.3) {$\Gamma'$};
\node[ext] (e1) at (.3,-.4) {\;};
\node[ext] (e2) at (.6,-.4) {\;};
\node at (.98,-.4) {$\scriptstyle\dots$};
\node[ext] (e3) at (1.3,-.4) {\;};
\draw (v) edge (ee1) edge (ee2) edge (w);
\draw (w) edge (e1) edge (e2) edge (e3);
\end{tikzpicture}
-(-1)^{|\Gamma||\Gamma'|}
\sum
\begin{tikzpicture}[baseline=-.8ex]
\node[draw,circle] (v) at (0,1) {$\Gamma'$};
\node[ext] (ee1) at (-.5,.3) {\;};
\node[ext] (ee2) at (-.2,.3) {\;};
\node at (.18,.3) {$\scriptstyle\dots$};
\node[draw,circle] (w) at (.8,.3) {$\Gamma$};
\node[ext] (e1) at (.3,-.4) {\;};
\node[ext] (e2) at (.6,-.4) {\;};
\node at (.98,-.4) {$\scriptstyle\dots$};
\node[ext] (e3) at (1.3,-.4) {\;};
\draw (v) edge (ee1) edge (ee2) edge (w);
\draw (w) edge (e1) edge (e2) edge (e3);
\end{tikzpicture},
\]
where the sum runs over all external vertices of one graph and over all ways of attaching its edge to internal vertices of another graph. $|\Gamma|$ is the degree of $\Gamma$.

Furthermore, for $n=m$ there is a Maurer-Cartan element
\[
m_0 := 
\begin{tikzpicture}[baseline=-.65ex]
\node[ext] (h1) at (0,0) {\;};
\node[ext] (h2) at (0.7,0) {\;};
\draw (h1) edge (h2);
\end{tikzpicture}
\in \HGC_{n,n},
\]
that can be used to twist the complex to $(\HGC_{n,n},\delta+[m_0,\cdot])$.
There is morphism of complexes
\[
(\GC_n[-1],\delta) \to \left(\HGC_{n,n},\delta+[m_0,-]\right)
\]
between the non-hairy graph complex $\GC_n$ and the twisted complex obtained by attaching one hair to a non-hairy graph.
More precisely, a graph $\Gamma \in \GC_n[-1]$ is sent to the linear combination
\begin{align*}
\sum_v \Gamma_v \in \HGC_{n,n},
\end{align*}
where the sum is over the vertices of $\Gamma$ and $\Gamma_v$ is obtained from $\Gamma$ by attaching a hair at vertex $v$.
The following result can be found in \cite{grt,TW2}.
\begin{prop}\label{prop:GCHGC}
The above map of complexes $(\GC_n[-1],\delta) \to \left(\HGC_{n,n},\delta+[m_0,-]\right)$ induces an isomorphism in cohomology in loop orders $\geq 2$.
\end{prop}

\subsection{Complexes of directed acyclic non-hairy graphs}

In this and the next subsection we define, or recall the definition of several complexes of directed acyclic graphs.
We call these complexes "oriented graph complexes" to comply with the notation of the literature, and also to avoid confusion with the unrelated meaning of the term "acyclic" in homological algebra.

Here we will define the complex of oriented non-hairy graphs. It is considered e.g.\ in \cite{Woriented} and \cite{MultiOriented}.

\subsubsection{Combinatorial description}

Consider the set $\bar\mV_v\bar\mE_e\mO\grac^{2}$ containing directed graphs that
\begin{itemize}
\item are connected;
\item have $v>0$ distinguishable vertices that are at least 2-valent;
\item have $e\geq 0$ distinguishable directed edges;
\item have no passing vertices (2-valent vertices with one incoming and one outgoing edge $\EEpassing$);
\item have no closed directed path along the directed edges (\emph{directed cycles}).
\end{itemize}

For $n\in\Z$, let
\begin{equation}
\bar\mV_v\bar\mE_e\mO\G_n:=\left\langle\bar\mV_v\bar\mE_e\mO\grac^{2}\right\rangle[(n-1)e-nv+n]
\end{equation}
be the vector space of degree shifted formal linear combinations.

Unlike for non-oriented graphs, here we want to keep the direction of edges, i.e.\ we will not take the space of invariants under the action of changing the direction of an edge.
Therefore, let us consider the action of the group $\sym_v\times\sym_e$ on $\bar\mV_v\bar\mE_e\bar\mH_S\mO\grac^{2}$. Let
\begin{equation}
\mV_v\mE_e\mO\G_{n}:=\left\{
\begin{array}{ll}
\left(\bar\mV_v\bar\mE_e\mO\G_n\otimes\sgn_e\right)^{\sym_v\times\sym_e}
\qquad&\text{for $n$ even,}\\
\left(\bar\mV_v\bar\mE_e\mO\G_n\otimes\sgn_v\right)^{\sym_v\times\sym_e}
\qquad&\text{for $n$ odd,}
\end{array}
\right.
\end{equation}

The underlying vector space of 
the \emph{oriented graph complex} is given by
\begin{equation}
\OG_{n}:=\prod_{v\geq 1,e\geq 0}\mV_v\mE_e\mO\G_{n}.
\end{equation}

The differential again acts by edge contraction:
\begin{equation}
d(\Gamma)=\sum_{a\in E(\Gamma)}\Gamma/a
\end{equation}
where $E(\Gamma)$ is the set of edges of $\Gamma$. If a directed cycle is produced, we consider the result to be zero.

The dual complex is
\begin{equation}
\left(\OGC_{n},\delta\right):=\left(\OG_{n},d\right)^*.
\end{equation}
Here differential $\delta$ acts by splitting a vertex.

\subsubsection{Comparison of complexes of oriented and undirected non-hairy graphs}
\label{sec:GCOGCcomparison}

The following results shows that the graph complexes $\G_{n}$ and $\OG_{n+1}$ are homologicaly essentially the same.

\begin{thm}[\cite{MultiOriented}, \cite{MultiSourced}]
For every $n\in\Z$ there is a morphism of complexes
$$
h:\left(\G_{n},d\right)\to \left(\OG_{n+1},d\right)
$$
that respects the gradings by loop order and that induces an isomorphism on cohomology in loop orders $\geq 2$.
\end{thm}

The map from the theorem is defined as
\begin{equation}\label{equ:GOGmapdef}
h(\Gamma):=
\frac 1 {2g}\sum_{x\in V(\Gamma)}(v(x)-2)\sum_{\tau\in S(\Gamma)}h_{x,\tau}(\Gamma),
\end{equation}
where sums go through all vertices $x$ of $\Gamma$ and all spanning trees $\tau$ of $\Gamma$, $v(x)$ is the valence of $x$ and $g$ is the loop order (the first Betti number) of the graph.

The graph $h_{x,\tau}(\Gamma)$ is the oriented graph obtained from $\Gamma$ by giving to edges of $\tau$ the direction away from the vertex $x$, and replacing other edges with the structure $\EdE$. Detailed construction can be found in \cite{MultiOriented} or \cite{MultiSourced}, and it is similar to our construction from Subsection \ref{ss:map}. The biggest difference is that here we have the extra summation over all vertices. We note that we have included a conventional prefactor $\frac 1 {2g}$ here, that is not present in \cite{MultiOriented, MultiSourced}. Recall that $g=e-v+1$ is the loop order.

\subsection{Complexes of oriented hairy graphs} \label{sec:HOG}

\subsubsection{Combinatorial description}\label{sec:HOGcombinatorial}

Consider the set $\bar\mV_v\bar\mE_e\bar\mH_S\mO\grac^{2}$ containing directed graphs that
\begin{itemize}
\item are connected;
\item have $v>0$ distinguishable internal vertices that are at least 2-valent;
\item have $e\geq 0$ distinguishable directed edges;
\item have $|S|\geq 0$ distinguishable external vertices labelled by the set $S$ that are 1-valent targets, i.e.\ they have one incoming edge attached;
\item have no internal targets (internal vertices without outgoing edge);
\item have no passing vertices (2-valent vertices with one incoming and one outgoing edge $\EEpassing$);
\item have no closed directed path along the directed edges (\emph{directed cycles}).
\end{itemize}
For some pictures of such graphs see \eqref{equ:HOGexample}.

For $n\in\Z$, let
\begin{equation}
\bar\mV_v\bar\mE_e\bar\mH_S\mO\G_n:=\left\langle\bar\mV_v\bar\mE_e\bar\mH_S\mO\grac^{2}\right\rangle[(n-1)e-nv]
\end{equation}
be the vector space of degree shifted formal linear combinations.

Consider the action of the group $\sym_v\times\sym_e$ on $\bar\mV_v\bar\mE_e\bar\mH_S\mO\grac^{2}$.
Let
\begin{equation}
\mV_v\mE_e\bar\mH_S\mO\G_{n}:=\left\{
\begin{array}{ll}
\left(\bar\mV_v\bar\mE_e\bar\mH_S\mO\G_n\otimes\sgn_e\right)^{\sym_v\times\sym_e}
\qquad&\text{for $n$ even,}\\
\left(\bar\mV_v\bar\mE_e\bar\mH_S\mO\G_n\otimes\sgn_v\right)^{\sym_v\times\sym_e}
\qquad&\text{for $n$ odd,}
\end{array}
\right.
\end{equation}

The underlying vector space of 
the \emph{oriented hairy graph complex} is given by
\begin{equation}
\HOGS_{n}:=\prod_{v\geq 1,e\geq 0}\mV_v\mE_e\bar\mH_S\mO\G_{n}.
\end{equation}

The differential again acts by edge contraction:
\begin{equation}
d(\Gamma)=\sum_{a\in E(\Gamma)}\Gamma/a
\end{equation}
where $E(\Gamma)$ is the set of edges of $\Gamma$ that are not connected to an external vertex.
If a directed cycle is produced, we consider the result to be zero.

The dual complex is
\begin{equation}
\left(\HOGCS_{n},\delta\right):=\left(\HOGS_{n},d\right)^*.
\end{equation}
Here differential $\delta$ acts by splitting an internal vertex.

%
%

\subsubsection{A version with input hairs}
We will also need to consider a slight variant $\HHOGS_{n}$ of the complex $\HOGS_{n}$ above, obtained by changing the definition as follows:
\begin{itemize}
	\item A graph must have in addition to the output hairs labelled by $S$ an arbitrary (positive) number of input hairs.
	\item The graphs must not have sources, i.e., internal vertices with only outgoing edges.
\end{itemize}
Here is an example:
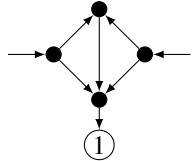
\[
\begin{tikzpicture}[scale=.6, yshift=.5cm]
\node[int] (v1) at (0:1) {};
\node[int] (v2) at (90:1) {};
\node[int] (v3) at (180:1) {};
\node[int] (v4) at (-90:1) {};
\node[ext] (h1) at (-90:2) {1};
\draw[-latex] (v2) edge (v4) (v1) edge (v4) edge (v2) (v3) edge (v4) edge (v2) (v4) edge (h1);
\draw[latex-] (v1) edge +(1,0) (v3) edge +(-1,0); 
\end{tikzpicture}
\]

Consider the two types of 'special' vertices $v$ for a graph $\Gamma \in \HHOGS_n$,
\begin{equation}\label{eq:HHGCspecialvert}
\begin{tikzpicture}[scale=.6, yshift=1cm]
\node[] (v1) at (-1.5,0) {$\Gamma\setminus v$};
\node[int,label=-0:{$\scriptstyle v$}] (v2) at (0,-1.5) {};
\node[ext] (h1) at (0,-3) {r};

\draw[-latex] (v1) edge (v2) (v2) edge (h1); 

\draw[latex-] (v2) edge +(0,1) (v2) edge +(1,1); 

\node[] (v2) at (0.5,-0.5) {$_{\ldots}$}; 
\end{tikzpicture}, \quad \text{and}\quad    \begin{tikzpicture}[scale=.6, yshift=-1cm]
\node[] (v1) at (0,0) {$\Gamma\setminus v$};
\node[int,label=-0:{$\scriptstyle v$}] (v2) at (0 ,1.5) {};

\draw[latex-] (v1) edge (v2); 

\draw[latex-] (v2) edge +(-0.6,1)  (v2) edge +(-0.3,1) (v2) edge +(0.6,1); 

\node[] (v2) at (0.2,2.5) {$_{\ldots}$};
\end{tikzpicture},
\end{equation}
with either one ingoing internal edge, one outgoing hair and an arbitrary number of ingoing hairs, or one internal outgoing edge and an arbitrary number of ingoing hairs.

Let $d_v$ be the operation of removing such special vertices 
$$d_v\left( \begin{tikzpicture}[scale=.6, yshift=1cm]
\node[] (v1) at (-1.5,0) {$\Gamma\setminus v$};
\node[int,label=-0:{$\scriptstyle v$}] (v2) at (0,-1.5) {};
\node[ext] (h1) at (0,-3) {r};

\draw[-latex] (v1) edge (v2) (v2) edge (h1); 

\draw[latex-] (v2) edge +(0,1) (v2) edge +(1,1); 

\node[] (v2) at (0.5,-0.5) {$_{\ldots}$};
\end{tikzpicture}  \right) :=  \begin{tikzpicture}[scale=.6, yshift=1cm]
\node[] (v1) at (-1.5,0) {$\Gamma\setminus v$};
\node[ext] (v2) at (0,-1.5) {$r$};

\draw[-latex] (v1) edge (v2) ; 

\end{tikzpicture}, \quad\text{and}\quad d_v\left(\begin{tikzpicture}[scale=.6, yshift=-1cm]
\node[] (v1) at (0,0) {$\Gamma\setminus v$};
\node[int,label=-0:{$\scriptstyle v$}] (v2) at (0 ,1.5) {};

\draw[latex-] (v1) edge (v2); 

\draw[latex-] (v2) edge +(-0.6,1)  (v2) edge +(-0.3,1) (v2) edge +(0.6,1); 

\node[] (v2) at (0.2,2.5) {$_{\ldots}$};
\end{tikzpicture}\right):= \begin{tikzpicture}[scale=.6, yshift=-1cm]
\node[] (v1) at (0,0) {$\Gamma\setminus v$};
\node[] (v2) at (0 ,1.5) {};

\draw[latex-] (v1) edge (v2); 
\end{tikzpicture}.$$
The differential $d$ on $\HHOGS_n$ acts by
\begin{equation}\label{equ:donHHOGC}
	d(\Gamma)= \sum_{e\in E(\Gamma)} \Gamma/e - \sum_{v\in V(\Gamma)} d_v(\Gamma).
\end{equation}

Dually, we again define 
\[
\left(\HHOGCS_{n},\delta\right) := \left(\HHOGS_{n},d\right)^*.
\]

It is clear that the differential cannot change the loop order in graphs. Hence we get in particular a splitting of complexes
\[
\left(\HHOGCS_{n},\delta\right)
\cong 
\prod_g  \left(B_g\HHOGCS_{n},\delta\right),
\]
where we denote by 
\[
	B_g\HHOGCS_{n} \subset \HHOGCS_{n}
\]
the loop order $g$ subcomplex.

\subsubsection{Description as properadic deformation complex}\label{sec:HHOGCSprop}
The oriented graph complexes are very closely connected to properadic deformation complexes. For example, it has been shown in \cite{MWDef} that the complex $\OGC_n$ computes (essentially) the homotopy derivations of a degree shifted version of the Lie bialgebra properad. We can also identify the complexes $\HHOGCS_{n}$ above with pieces of properadic deformation complexes. For our purposes here this reformulation has the main advantage that we do not have to pay too close attention to combinatorial signs and prefactors, which are automatically handled due to generalities on deformation complexes.

For general definitions and statements about deformation complexes we refer to \cite{MV-properad} and \cite[section 3]{MWDef}, whose conventions we shall follow.
Let us only mention that if $\op C$ is a cooperad, which we assume reduced in the sense that $\op C(1,1)=\Q$, one can define a properad $\Omega(\op C)$ via the properadic cobar construction. Furthermore, if $\op P$ is another properad, then we can endow the graded vector space
\[
\prod_{r,s}\geq 1 \Hom_{\bbS_r\times \bbS_r} \left( \op C(r,s), \op P(r,s)\right)	
\] 
with a dg Lie algebra structure, in such a way that the Maurer-Cartan elements are (essentially) in one-to-one correspondence with dg properad maps $\Omega(\op C)\to \op P$. Then, given a properad map $f:\Omega(\op C)\to \op P$ we define the deformation complex $\Def(\Omega(\op C) \xrightarrow{f} \op P)$ as the twist of the dg Lie algebra above by the Maurer-Cartan element associated to the map $f$.

Now consider a properad $\AC$ such that, for $r,s\geq 1$
\[
\AC(r,s) := \Ass(r)\otimes\Com(s) \cong \Ass(r)\cong \K[\bbS_r]. 
\]
The composition morphisms
\[
\circ_j : \AC(r_1,s_1) \otimes \AC(r_2,s_2) \to \AC(r_1+r_2-1,s_1+s_2-1)
\]
are defined as follows. 
\begin{itemize}
\item If $r_1\geq 2$ and $r_2\geq 2$ then $\circ_j$ above is the zero morphism.
\item If $r_1=1$ or $r_2=1$ then $\AC(r_1,s_1)=\K$ or $\AC(r_2,s_2)=\K$ and we define the composition $\circ_j$ to be the identity morphism.
\end{itemize} 
We furthermore define all "higher genus" compositions to be zero.

The operadic ($r=1$-)part of the properad $\AC$ is the commutative operad $\Com$. In particular, we have maps of properads
\[
\Com \to \AC \to \Com.
\]
We also have properad maps 
\[
\IFrob \xrightarrow{*} \AC \xrightarrow{*} \IFrob
\]
factoring through $\Com$. (I.e., the maps are zero in output arity $r\geq 2$.).
Taking properadic (co)bar constructions we hence get properad maps 
\[
 \La\LieB_\infty =\Omega(\IFrob^*) \to \Omega(\AC^*) \to \Omega(\IFrob^*) = \La\LieB_\infty.
\]
Consider the resulting deformation complex (disregarding the differential for now)
\begin{align*}
\Def\left(\Omega(\AC^*) \xrightarrow{*} \La\LieB_\infty\right)
&=
\prod_{r,s}\Hom_{\bbS_r\times \bbS_s}\left(\AC^*(r,s),\La\LieB_\infty(r,s) \right)
\\&\cong 
\prod_{r,s} \left( \La\LieB_\infty(r,s) \right)^{\bbS_s}
.
\end{align*}
Elements of $\La\LieB_\infty(r,s)$ can graphically be considered as linear combinations of directed acyclic graphs with $s$ numbered inputs and $r$ numbered outputs.
Hence the $(r,s)$ piece of the above product is given by linear combinations of drected acyclic graphs with $s$ unlabeled inputs and $r$ numbered outputs. Furthermore, one can check that the differential on the deformation complex is combinatorially just the dual version of the operation \eqref{equ:donHHOGC}: The internal differential on $\La\Lie_\infty$ splits vertices and is dual to the first summand of \eqref{equ:donHHOGC}, and the twist by the Maurer-Cartan element corresponding to the map $*$ is dual to the second summand in \eqref{equ:donHHOGC}.
It is hence also clear that there are natural gradings by loop order of the graphs, and by the output arity $r$, and we can hence consider the subcomplex of loop order $g$ with $r$ outputs
\[
B_g\Def\left(\Omega(\AC^*) \xrightarrow{*} \La\LieB_\infty\right)^r \subset \Def\left(\Omega(\AC^*) \xrightarrow{*} \La\LieB_\infty\right).
\]
Then for $g\geq 1$ we have an isomorphism of complexes
\begin{equation}\label{equ:BgDefBgHHOGC}
B_g\Def\left(\Omega(\AC^*) \xrightarrow{*} \La\LieB_\infty\right)^r
\cong 
B_g \HHOGC_{[r],1}.
\end{equation}

\begin{rem}$ $
\begin{itemize}
\item One may similarly construct a definition $\HHOGC_{[r],n}$ for $n\neq 1$ as well -- we shall however only use the case $n=1$ described above in this paper.
\item
Mind that we state \eqref{equ:BgDefBgHHOGC} for loop orders $g\geq 1$. In loop order zero, there is a minor conventional difference between both sides, in that the unit is an element in $\La\LieB_\infty(1,1)$ whereas the corresponding graph with no internal vertex is not contained in $\HHOGC_{[r],1}$ according to our conventions.

Furthermore, one cannot immediately remove the "$B_g$" on both sides of \eqref{equ:BgDefBgHHOGC}, because $\HHOGC_{[r],1}$ contains potentially infinite series of diagrams whereas $\La\LieB_\infty$ only contains linear combinations. 
This slight technical difficulty could be countered by taking the completion of $\La\LieB_\infty$ by loop order, as is done in \cite{MWDef}.
\end{itemize}
\end{rem}

\subsubsection{Comparison of both versions}
The following result shows that the graphs complexes $\HOGS_{n}$ and $\HHOGS_{n}$ can be considered as "the same object" homologically.

\begin{prop}
For every $S,n$ the map
\begin{align*}
\HHOGS_{n} \to \HOGS_{n}
\end{align*} 
obtained by deleting all input legs from a graph induces an isomorphism in cohomology.
\end{prop}
\begin{proof}
	The result is essentially \cite[Proposition 4.1.2]{MWDef}, expcept that there the output hairs are not numbered. However, numbering the output legs does not affect the proof at all, so that one can obtain our result above just by copy-pasting the argument.	
\end{proof}

\subsubsection{Complexes with (anti-)symmetrized hairs}

As in the non-directed case, we may (anti-)symmetrize external vertices. Let $S=\{1,\dots,s\}$ and let $\sym_s$ act on $\HOGS_n$ by permuting external vertices. For an integer $m$ let
\begin{equation}
\HOGs_{m,n}:=\left\{
\begin{array}{ll}
\left(\HOGS_n\otimes\Q[-m]^{\otimes s}\right)^{\sym_s}[m]
\qquad&\text{for $m$ even,}\\
\left(\HOGS_n\otimes\Q[-m]^{\otimes s}\otimes\sgn_s\right)^{\sym_s}[m]
\qquad&\text{for $m$ odd.}
\end{array}
\right.
\end{equation}
\begin{equation}
\HOG_{m,n} = \bigoplus_{s>0} \HOGs_{m,n}
\end{equation}
The differential $d$ is still contracting an edge. There is a dual
\begin{equation}
\left(\HOGC_{m,n},\delta\right)=\left(\HOG_{m,n},d\right)^*.
\end{equation}

There is a similar Lie algebra structure on $\HOGC_{m,n}$ defined as
\[
\left[ 
\begin{tikzpicture}[baseline=-.8ex]
\node[draw,circle] (v) at (0,.3) {$\Gamma$};
\node[ext] (e1) at (-.5,-.4) {\;};
\node[ext] (e2) at (-.2,-.4) {\;};
\node at (.18,-.4) {$\scriptstyle\dots$};
\node[ext] (e3) at (.5,-.4) {\;};
\draw (v) edge[-latex] (e1) edge[-latex] (e2) edge[-latex] (e3);
\end{tikzpicture}
,
\begin{tikzpicture}[baseline=-.65ex]
\node[draw,circle] (v) at (0,.3) {$\Gamma'$};
\node[ext] (e1) at (-.5,-.4) {\;};
\node[ext] (e2) at (-.2,-.4) {\;};
\node at (.18,-.4) {$\scriptstyle\dots$};
\node[ext] (e3) at (.5,-.4) {\;};
\draw (v) edge[-latex] (e1) edge[-latex] (e2) edge[-latex] (e3);
\end{tikzpicture}
\right]
=
\sum
\begin{tikzpicture}[baseline=-.8ex]
\node[draw,circle] (v) at (0,1) {$\Gamma$};
\node[ext] (ee1) at (-.5,.3) {\;};
\node[ext] (ee2) at (-.2,.3) {\;};
\node at (.18,.3) {$\scriptstyle\dots$};
\node[draw,circle] (w) at (.8,.3) {$\Gamma'$};
\node[ext] (e1) at (.3,-.4) {\;};
\node[ext] (e2) at (.6,-.4) {\;};
\node at (.98,-.4) {$\scriptstyle\dots$};
\node[ext] (e3) at (1.3,-.4) {\;};
\draw (v) edge[-latex] (ee1) edge[-latex] (ee2) edge[-latex] (w);
\draw (w) edge[-latex] (e1) edge[-latex] (e2) edge[-latex] (e3);
\end{tikzpicture}
-(-1)^{|\Gamma||\Gamma'|}
\sum
\begin{tikzpicture}[baseline=-.8ex]
\node[draw,circle] (v) at (0,1) {$\Gamma'$};
\node[ext] (ee1) at (-.5,.3) {\;};
\node[ext] (ee2) at (-.2,.3) {\;};
\node at (.18,.3) {$\scriptstyle\dots$};
\node[draw,circle] (w) at (.8,.3) {$\Gamma$};
\node[ext] (e1) at (.3,-.4) {\;};
\node[ext] (e2) at (.6,-.4) {\;};
\node at (.98,-.4) {$\scriptstyle\dots$};
\node[ext] (e3) at (1.3,-.4) {\;};
\draw (v) edge[-latex] (ee1) edge[-latex] (ee2) edge[-latex] (w);
\draw (w) edge[-latex] (e1) edge[-latex] (e2) edge[-latex] (e3);
\end{tikzpicture},
\]
where the sum runs over all external vertices of one graph and over all ways of attaching its edge to internal vertices of another graph.

Here, for $m=n-1$ there is a Maurer-Cartan element
\[
m_1 =
\sum_{k\geq 2} 
\frac 2 {k!}
\underbrace{
\begin{tikzpicture}[baseline=-.65ex]
\node[int] (v) at (0,.5) {};
\node[ext] (h1) at (-.7,-.2) {\;};
\node[ext] (h2) at (-.3,-.2) {\;};
\node[ext] (h3) at +(.7,-.2) {\;};
\draw (v) edge[-latex] (h1) edge[-latex] (h2)  edge[-latex] (h3) +(.25,-.7) node {$\scriptstyle \dots$};
\end{tikzpicture}
}_{k\times}
\in \HOGC_{n,n+1}.
\]
that can be used to twist the complex to $(\HOGC_{n,n+1},\delta+[m_1,\cdot])$.

\subsubsection{Comparison of complexes of oriented hairy and non-hairy graphs}\label{sec:hairynonhairOGC}

We recall from \cite{MWDef} that one can define a map of complexes between the hairy and nonhairy oriented graph complexes
\[
 (\OGC_n,\delta) \to \left(\HOGC_{n-1,n}, \delta + [m_1,\cdot]\right)
\]
by sending a graph $\Gamma\in \OGC_n$ to the infinite series of graphs obtained by attaching hairs to vertices in all possible ways, schematically
\[
\Gamma \mapsto \sum_{k\geq 1} \frac 1 {k!} 
\underbrace{
\begin{tikzpicture}
	\node[] (v) at (0,.5) {$\Gamma$};
	\node[ext] (h1) at (-.7,-.2) {\;};
	\node[ext] (h2) at (-.3,-.2) {\;};
	\node[ext] (h3) at +(.7,-.2) {\;};
	\draw (v) edge[-latex] (h1) edge[-latex] (h2)  edge[-latex] (h3) +(.25,-.7) node {$\scriptstyle \dots$};
\end{tikzpicture}
}_{k\times}
\in \HOGC_{n,n+1}
\]

If the graphs produced on the right have sinks, thus violating the conditions of section \ref{sec:HOGcombinatorial}, they are dropped. So at least one hair needs to be added to each sink of $\Gamma$. Furthermore, it is clear that the map respects the loop order gradings on both sides. From \cite{MWDef} we cite the following result.
\begin{prop}[Proposition 3 of \cite{Woriented} or Proposition 4.1.1 of \cite{MWDef}]\label{prop:OGCtoHOGC}
The map $ (\OGC_n,\delta) \to (\HOGC_{n-1,n}, \delta + [m_1,\cdot])$ above is a quasi-isomorphism in loop orders $\geq 1$.
\end{prop}

We shall elaborate a bit further on the result.
Every oriented non-hairy graph needs to have a target, and we may construct a filtration of the complex on the number of targets. The differential graded complex, called \emph{fixed target graph complex}, is $(\OGC_n,\delta_0)$ where $\delta_0$ is the part of the differential that does not change the number of targets (sinks). It, or rather its isomorphic version with sources instead of targets, is considered in detail in \cite{AssarMarko}.

There associated graded of the above map is then
\begin{equation}
\label{eq:incl}
(\OGC_n[-1],\delta_0)\hookrightarrow (\HOGC_{n-1,n},\delta)
\end{equation}
where a hair is attached to every target.
Proposition \ref{prop:OGCtoHOGC} can then be seen as an immediate Corollary of the following result.

\begin{prop}
\label{prop:incl}
For every $n$ the inclusion
$$
(\OGC_n[-1],\delta_0)\hookrightarrow (\HOGC_{n-1,n},\delta)
$$
is a quasi-isomorphism.
\end{prop}
\begin{proof}
Let $\Gamma\in\HOGC_{n-1,n}$. An internal vertex in $\Gamma$ is called ``bad vertex'' if it shares an edge with an external vertex and it has more than one outgoing edges. The one going to the external vertex has to be outgoing.

The image of the inclusion is exactly the sub-complex spanned by graphs with no bad vertices. It is enough to prove that the quotient spanned by graphs with at least one bad vertex is acyclic.

On that quotient let us make a spectral sequence on the number of internal vertices that are not 2-valent bad vertices. Those vertices are 2-valent sources $\dEd$ with one edge heading towards an external vertex. On the first page of the spectral sequence there is the differential $\delta_{1}$ that produces such a vertex. It is produced by splitting a bad vertex that was more than 2-valent.

There is a homotopy $h$ that contracts the edge of a 2-valent bad vertex that does not head towards the external vertex. One easily checks that $h\delta_1+\delta_1 h=c\id$ where $c$ is the number of bad vertices. This implies that the cohomology on the first page of the spectral sequence is zero. After splitting the complexes into complexes with fixed loop order, standard spectral sequence arguments imply that the spectral sequence converges correctly, so the result follows.
\end{proof}

\subsection{Skeleton version of $\HOG$}
For later proofs it will be convenient to consider an alternative definition of the 
complex of directed acyclic graphs $\left(\HOGS_n,d\right)$ introduced above.
More precisely, let us define the graph complex $\left(\HOsGS_n,d\right)\cong \left(\HOGS_n,d\right)$ as follows, using results from from \cite[Section 2]{MultiSourced}: 

Recall the set $\bar\mV_v\bar\mE_e\bar\mH_S\grac^{3}$ from Subsection \ref{sss:DefUndirected}.
In this context, those graphs are called \emph{core graphs}. The direction of an edge in a core graph is called \emph{core direction}.
To each edge of a core graph we attach an ``edge type'' from $\sigma:=\{\Ed[n-1],\dE[n-1],\Ess[n-2]\}$. The direction of elements of $\sigma$ are called \emph{type directions}. For $\Ed$ and $\Ess$ they go along core direction, and for $\dE$ it is the opposite of core direction.
Admissible graphs are those graphs with edge types such that:
\begin{itemize}
\item every internal vertex has an attached edge of type $\Ed$ or $\dE$ with type direction away from the vertex (they are not type-targets);
\item an edge adjacent to an external vertex is not of type $\Ed$ or $\dE$ with type direction going away from the external vertex (external vertices are type-targets); 
\item there are no closed paths along edges of type $\Ed$ or $\dE$ along type directions (\emph{type-directed cycles}).
\end{itemize}
The set of all admissible graphs is denoted by $\bar\mV_v\bar\mE_e\bar\mH_S\mO^{sk}\grac\subset\bar\mV_v\bar\mE_e\bar\mH_S\grac^{3}\times\sigma^{\times e}$.
For $n\in\Z$, let
\begin{equation}
\bar\mV_v\bar\mE_e\bar\mH_S\mO^{sk}\G_n:=\left\langle\bar\mV_v\bar\mE_e\bar\mH_S\mO^{sk}\grac\right\rangle[-nv].
\end{equation}
Note that admissible graphs already have degrees that come from degree of edge types in $\sigma$.

Similarly as before, there is a natural right action of the group $\sym_v\times \left(\sym_e\ltimes \sym_2^{\times e}\right)$ on $\bar\mV_v\bar\mE_e\bar\mH_S\mO^{sk}\G_n$, where $\sym_v$ permutes vertices, $\sym_e$ permutes edges, and $\sym_2^{\times e}$ changes the core direction of edges and changes edge types as
\begin{equation}
\Ed \leftrightarrow\dE,\quad\Ess\mapsto -(-1)^{n}\Ess.
\end{equation}

Recall that $\sgn_v$ and $\sgn_e$ are one-dimensional representations of $\sym_v$, respectively $\sym_e$, where the odd permutation reverses the sign. They can be considered as representations of the whole product $\sym_v\times \left(\sym_e\ltimes \sym_2^{\times e}\right)$. Let
\begin{equation}
\label{eq:VEHOsG}
\mV_v\mE_e\bar\mH_S\mO^{sk}\G_{n}:=\left\{
\begin{array}{ll}
\left(\bar\mV_v\bar\mE_e\bar\mH_S\mO^{sk}\G_n\otimes\sgn_e\right)^{\sym_v\times \left(\sym_e\ltimes \sym_2^{\times e}\right)}
\qquad&\text{for $n$ even,}\\
\left(\bar\mV_v\bar\mE_e\bar\mH_S\mO^{sk}\G_n\otimes\sgn_v\right)^{\sym_v\times \left(\sym_e\ltimes \sym_2^{\times e}\right)}
\qquad&\text{for $n$ odd,}
\end{array}
\right.
\end{equation}
\begin{equation}
\HOsGS_{n}:=\prod_{v\geq 1,e\geq 0}\mV_v\mE_e\bar\mH_S\mO^{sk}\G_{n}.
\end{equation}

There are two differentials. The core differential $d_C$ contracts edges of type $\Ed$ and $\dE$ that connect two internal vertices. The edge differential $d_E$ changes type of edges as
\begin{equation}\label{equ:dEdef}
\Ess\mapsto\Ed-(-1)^n\dE,
\end{equation}
summed over all edges, including those connecting external vertices. If such operation produces a type-directed cycle or makes an external edge a type-target, we consider the result to be zero.
The total differential is
\begin{equation}
d:=d_C+(-1)^{n\deg}d_E.
\end{equation}

This complex $\left(\HOsGS_{n},d\right)$ is designed to be isomorphic to the original version $\left(\HOGS_{n},d\right)$. In short, at least 3-valent  vertices and 2-valent sources in a graph in $\HOGS_{n}$ are called \emph{skeleton internal vertices}. Strings of edges and vertices between two skeleton internal vertices or external vertices have to be in the set $\{\Ed,\dE,\dEd\}$, and they are called \emph{skeleton edges}. A corresponding graph in $\HOsGS_{n}$ is the one with skeleton internal vertices as internal vertices, external vertices remaining the same, and skeleton edges as edges, where $\dEd$ is mapped to $\Ess$. One can check that the degrees and parities are correctly defined, and obtain the following result.

\begin{prop}
\label{prop:sk}
There is an isomorphism of complexes
\begin{equation*}
\kappa:\left(\HOGS_{n},d\right)\rightarrow\left(\HOsGS_{n},d\right).
\end{equation*}
\end{prop}

It is probably easiest to illustrate the above correspondence by an example.
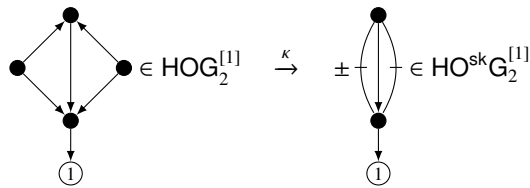
\begin{align*}
\begin{tikzpicture}[scale=.7]
	\node[int] (v1) at (0,1) {};
	\node[int] (v2) at (-1,0) {};
	\node[int] (v3) at (0,-1) {};
	\node[int] (v4) at (1,0) {};
	\node[ext] (w) at (0,-2) {$\scriptstyle 1$};
	\draw[-latex] (v1) edge (v3) (v2) edge (v1) edge (v3) (v4) edge (v1) edge (v3) (v3) edge (w);
\end{tikzpicture}
\in \HOG^{[1]}_{2}
\quad
\xrightarrow{\kappa}
\quad
\pm \,
\begin{tikzpicture}[scale=.7]
	\node[int] (v1) at (0,1) {};
	\node[int] (v3) at (0,-1) {};
	\node[ext] (w) at (0,-2) {$\scriptstyle 1$};
	\draw[-latex] (v1) edge (v3) (v3) edge (w);
	\draw (v1) edge[crossed, bend left] (v3) (v1) edge[crossed, bend right] (v3);
\end{tikzpicture}
\in \HOsG^{[1]}_{2}
\end{align*}
On the right-hand side we have drawn the crossed edges without directions, since these directions are identified, up to sign, by taking the $\bbS_2$-invariants.
In words we can say that the map $\kappa$ just replaces all patterns $\dEd$ in the graph by a crossed edge 
$\Essu$ and otherwise leaves the graph the same. It is clear that this is an isomorphism, the inverse map just performs that replacement in the opposite direction.
Note however that the resulting graphs' internal vertices are always at least 3-valent.

\section{The relation of oriented and undirected graph complexes}\label{sec:main proof}
In this section we show our main technical Theorems \ref{thm:main} and \ref{thm:mainlie}.

\subsection{The definition of the map}
\label{ss:map}

In this subsection we are going to define the map $\Phi:\left(\HGS_{n},d\right)\rightarrow\left(\HOGS_{n+1},d\right)$.
We follow the methods from \cite{MultiSourced} and \cite{AssarMarko}.
Thanks to Proposition \ref{prop:sk} we may consider as well the "skeleton version" $\left(\HOsGS_{n+1},d\right)\cong \left(\HOGS_{n+1},d\right)$ of the oriented graph complex and define our map as
\begin{equation}
\Phi:\left(\HGS_{n},d\right)\rightarrow\left(\HOsGS_{n+1},d\right).
\end{equation}
The map
\begin{equation*}
\F: \left(\HOGCS_{n+1},\delta\right) \to \left(\HGCS_{n},\delta\right),
\end{equation*}
in Theorem \ref{thm:main} is then the map dual to $\Phi$.

\subsubsection{Forests}
Let us pick the number of internal vertices $v$, the number of edges $e$ and the set of external vertices $S$.
Let $\Gamma\in\bar\mV_v\bar\mE_e\bar\mH_S\G_{n}$ be a graph with said numbers of vertices, edges and hairs.

Let a \emph{forest} be any subgraph of $\Gamma$ that contains all its external vertices, that does not contain cycles (of any orientation), and all of whose connected components contain exactly one external vertex each. Let a \emph{spanning forest} be a forest that contains all vertices.
Let $F(\Gamma)$ be the set of all spanning forests of $\Gamma$.
An example of a spanning forest is given in Figure \ref{fig:span}.

\begin{figure}[H]
$$
\begin{tikzpicture}
 \node[int] (a) at (0,0) {};
 \node[int] (b) at (1,0) {};
 \node[int] (c) at (1,1) {};
 \node[int] (d) at (0,1) {};
 \node[int] (a1) at (-.5,-.5) {};
 \node[int] (b1) at (1.5,-.5) {};
 \node[int] (c1) at (1.5,1.5) {};
 \node[int] (d1) at (-.5,1.5) {};
 \node[int] (x) at (-1,.5) {};
 \node[int] (y) at (1.5,.5) {};
 \node[int] (z) at (2.2,.5) {};
 \node[ext] (e1) at (-1.6,.5) {1};
 \node[ext] (e2) at (-.9,-.9) {2};
 \node[ext] (e3) at (1.9,-.9) {3};
 \node[ext] (e4) at (1.9,1.9) {4};
 \draw (a) edge[red] (b);
 \draw (b) edge[dotted] (c);
 \draw (c) edge[dotted] (d);
 \draw (d) edge[dotted] (a);
 \draw (a1) edge[dotted] (b1);
 \draw (b1) edge[red] (y);
 \draw (y) edge[red] (c1);
 \draw (b1) edge[dotted] (z);
 \draw (z) edge[dotted] (c1);
 \draw (y) edge[red] (z);
 \draw (c1) edge[dotted] (d1);
 \draw (d1) edge[dotted] (a1);
 \draw (a) edge[red] (a1);
 \draw (b) edge[dotted] (b1);
 \draw (c) edge[red] (c1);
 \draw (d) edge[red] (d1);
 \draw (a1) edge[dotted] (x);
 \draw (x) edge[red] (d1);
 \draw (x) edge[red] (e1);
 \draw (a1) edge[red] (e2);
 \draw (b1) edge[red] (e3);
 \draw (c1) edge[dotted] (e4);
\end{tikzpicture}
$$
\caption{\label{fig:span}
An example of a hairy graph $\Gamma$ for $S=\{1,2,3,4\}$ and a spanning forest $\tau\in F(\Gamma)$. Edges of the forest are red, while other edges are dotted.}
\end{figure}
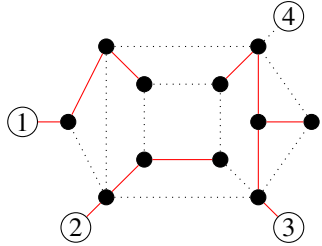

\subsubsection{Model pairs}
Let $\tau\in F(\Gamma)$ be a spanning forest of $\Gamma$. Also recall that $\Gamma$ comes with a numbering of edges and vertices. (These numberings are later removed in the definition of the graph complex, but we need to consider them here to define signs properly.)
We say that the pair $(\Gamma,\tau)$ is a \emph{model} if the following conditions are satisfied:
\begin{itemize}
	\item All edges of a connected component in $\tau$ are directed towards the external vertex of that connected component.
	\item An edge in $\tau$ has the same label as the vertex on its tail, labels being in the set $\{1,\dots,v\}$;
\end{itemize}

An example of model is given in Figure \ref{fig:model}.

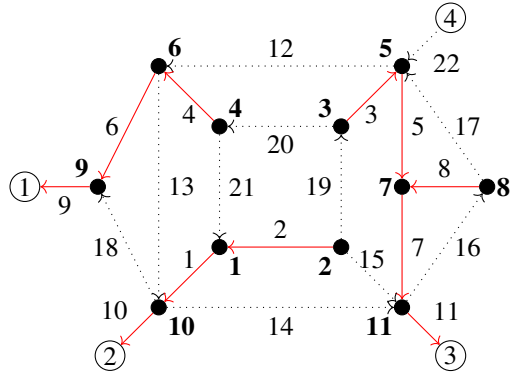
\begin{figure}[H]
$$
\begin{tikzpicture}[scale=1.6]
 \node[int] (a) at (0,0) {};
 \node[below right] at (a) {\bf 1};
 \node[int] (b) at (1,0) {};
 \node[below left] at (b) {\bf 2};
 \node[int] (c) at (1,1) {};
 \node[above left] at (c) {\bf 3};
 \node[int] (d) at (0,1) {};
 \node[above right] at (d) {\bf 4};
 \node[int] (a1) at (-.5,-.5) {};
 \node[below right] at (a1) {\bf 10};
 \node[int] (b1) at (1.5,-.5) {};
 \node[below left] at (b1) {\bf 11};
 \node[int] (c1) at (1.5,1.5) {};
 \node[above left] at (c1) {\bf 5};
 \node[int] (d1) at (-.5,1.5) {};
 \node[above right] at (d1) {\bf 6};
 \node[int] (x) at (-1,.5) {};
 \node[above left] at (x) {\bf 9};
 \node[int] (y) at (1.5,.5) {};
 \node[left] at (y) {\bf 7};
 \node[int] (z) at (2.2,.5) {};
 \node[right] at (z) {\bf 8};
 \node[ext] (e1) at (-1.6,.5) {1};
 \node[ext] (e2) at (-.9,-.9) {2};
 \node[ext] (e3) at (1.9,-.9) {3};
 \node[ext] (e4) at (1.9,1.9) {4};
 \draw (a) edge[red,<-] node[above] {\color{black} 2} (b);
 \draw (b) edge[dotted,->] node[left] {19} (c);
 \draw (c) edge[dotted,->] node[below] {20} (d);
 \draw (d) edge[dotted,->] node[right] {21} (a);
 \draw (a1) edge[dotted,->] node[below] {14} (b1);
 \draw (b1) edge[red,<-] node[right] {\color{black} 7} (y);
 \draw (y) edge[red,<-] node[right] {\color{black} 5} (c1);
 \draw (b1) edge[dotted,->] node[right] {16} (z);
 \draw (z) edge[dotted,->] node[right] {17} (c1);
 \draw (y) edge[red,<-] node[above] {\color{black} 8} (z);
 \draw (c1) edge[dotted,->] node[above] {12} (d1);
 \draw (d1) edge[dotted,->] node[right] {13} (a1);
 \draw (a) edge[red,->] node[above] {\color{black} 1} (a1);
 \draw (b) edge[dotted,->] node[above] {15} (b1);
 \draw (c) edge[red,->] node[below] {\color{black} 3} (c1);
 \draw (d) edge[red,->] node[below] {\color{black} 4} (d1);
 \draw (a1) edge[red,->] node[above left] {\color{black} 10} (e2);
 \draw (a1) edge[dotted,->] node[left] {18} (x);
 \draw (x) edge[red,<-] node[left] {\color{black} 6} (d1);
 \draw (x) edge[red,->] node[below] {\color{black} 9}(e1);
 \draw (b1) edge[red,->] node[above right] {\color{black} 11} (e3);
 \draw (c1) edge[dotted,<-] node[below right] {\color{black} 22}(e4);
\end{tikzpicture}
$$
\caption{\label{fig:model}
An example of model with the spanning forest from Figure \ref{fig:span}. Edges of the forest are red, while other edges are dotted. Labels of internal vertices are thick.}
\end{figure}

It is clear that every pair $(\Gamma,\tau)$ with $\tau \in F(\Gamma)$ can be mapped to a model by renumbering the edges and vertices and changing edge directions, i.e., by the action of an element of the group $\sym_v\times \left(\sym_e\ltimes \sym_2^{\times e}\right)$.



\subsubsection{Defining the map for a model}

Let us now pick up a model $(\Gamma,\tau)$, with a (single term) graph\footnote{Mind that elements of the graph complex are linear combinations of combinatorial graphs, and we consider here a single combinatorial graph, as an element of the graph complex.} $\Gamma\in\bar\mV_v\bar\mE_e\bar\mH_S\G_{n}$ and a spanning forest $\tau\in F(\Gamma)$. The graph $\Gamma$ after ignoring the degree can be considered as a core graph in $\bar\mV_v\bar\mE_e\bar\mH_S\grac^3$. To all of its edges that belong to the spanning forest (i.e., that are in $E(\tau)$) we attach an edge type $\Ed$, and to those that are not in the spanning forest we attach edge type $\Ess$ to get an element of $\bar\mV_v\bar\mE_e\bar\mH_S\mO^{sk}\grac$. Then after taking coinvariants and adding the degrees we get a skeleton graph
\begin{equation}
\Phi_{\tau}(\Gamma)\in\mV_v\mE_e\bar\mH_S\mO^{sk}\G_{n+1}.
\end{equation}

It is straightforward to check the following:
\begin{itemize}
\item the result $\Phi_{\tau}(\Gamma)$ is an admissible type oriented graph;
\item the map is well defined in a sense that if there is an element of $\sym_v\times \left(\sym_e\ltimes \sym_2^{\times e}\right)$ that sends one model to another model, the same result in $\mV_v\mE_e\mD^{sk}\GC_{n+1}$ is obtained;
\item $\Phi_{\tau}(\Gamma)$ and $\Gamma$ are of the same degree.
\end{itemize}

An example of $\Phi_{\tau}(\Gamma)$ is given in Figure \ref{fig:Phi}.

\begin{figure}[H]
$$
\begin{tikzpicture}[scale=1.4]
 \node[int] (a) at (0,0) {};
 \node[int] (b) at (1,0) {};
 \node[int] (c) at (1,1) {};
 \node[int] (d) at (0,1) {};
 \node[int] (a1) at (-.5,-.5) {};
 \node[int] (b1) at (1.5,-.5) {};
 \node[int] (c1) at (1.5,1.5) {};
 \node[int] (d1) at (-.5,1.5) {};
 \node[int] (x) at (-1,.5) {};
 \node[int] (y) at (1.5,.5) {};
 \node[int] (z) at (2.2,.5) {};
 \node[ext] (e1) at (-1.6,.5) {1};
 \node[ext] (e2) at (-.9,-.9) {2};
 \node[ext] (e3) at (1.9,-.9) {3};
 \node[ext] (e4) at (1.9,1.9) {4};
 \draw (a) edge[latex-] (b);
 \draw (b) edge[crossed,->] (c);
 \draw (c) edge[crossed,->] (d);
 \draw (d) edge[crossed,->] (a);
 \draw (a1) edge[crossed,->] (b1);
 \draw (b1) edge[latex-] (y);
 \draw (y) edge[latex-] (c1);
 \draw (b1) edge[crossed,->] (z);
 \draw (z) edge[crossed,->] (c1);
 \draw (y) edge[latex-] (z);
 \draw (c1) edge[crossed,->] (d1);
 \draw (d1) edge[crossed,->] (a1);
 \draw (a) edge[-latex] (a1);
 \draw (b) edge[crossed,->] (b1);
 \draw (c) edge[-latex] (c1);
 \draw (d) edge[-latex] (d1);
 \draw (a1) edge[crossed,->] (x);
 \draw (x) edge[latex-] (d1);
 \draw (x) edge[-latex] (e1);
 \draw (a1) edge[-latex] (e2);
 \draw (b1) edge[-latex] (e3);
 \draw (c1) edge[crossed,<-] (e4);
\end{tikzpicture}
$$
\caption{\label{fig:Phi}
Hairy oriented graph $\Phi_\tau(\Gamma)$ for the graph $\Gamma$ and spanning forest $\tau$ from  Figure \ref{fig:span}.}
\end{figure}
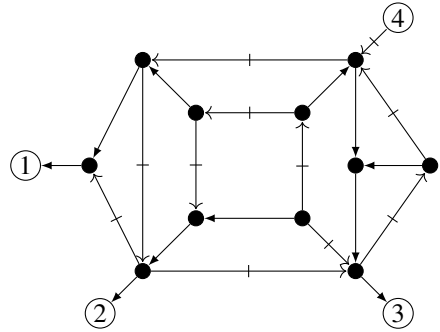

\subsubsection{The final map}

The map is now extended to all pairs $(\Gamma,\tau)$ by invariance under the action of $\sym_v\times \left(\sym_e\ltimes \sym_2^{\times e}\right)$. Then let us define
\begin{equation}
\Phi:\bar\mV_v\bar\mE_e\bar\mH_S\G_{n}\rightarrow\mV_v\mE_e\bar\mH_S\mO^{sk}\G_{n+1}, \quad
\Gamma\mapsto\sum_{\tau\in F(\Gamma)}\Phi_{\tau}(\Gamma).
\end{equation}
The invariance under all actions implies that the induced map $\Phi:\mV_v\mE_e\bar\mH_S\G_{n}\rightarrow\mV_v\mE_e\bar\mH_S\mO^{sk}\G_{n+1}$ is well defined. It is then extended to a map of graded vector spaces
\begin{equation}
\Phi:\HGCS_{n}\rightarrow\HOsGS_{n+1}.
\end{equation}

Rewording the above construction up to signs, this map is defined on some (undirected hairy) graph $\Gamma\in \HGCS_{n}$ as follows: We sum over all spanning forests of $\Gamma$. For each such forest 
we build a directed acyclic graph by the following procedure:
\begin{itemize}
\item We direct each edge in the spanning forest towards the unique external vertex in its tree.
\item We replace all other edges by a crossed edge 
$\Essu$.
Recall also that these crossed edges in the "skeleton" graph complex $\HOsGS_{n+1}$ are just placeholders for zigzags $\dEd$ in graphs in the graph complex $\HOGS_{n+1}$. 
\end{itemize}

\begin{prop}
\label{prop:map}
The map $\Phi:\left(\HGCS_{n},d\right)\rightarrow\left(\HOsGS_{n+1},d\right)$ is a map of complexes of degree zero, i.e.\
\begin{equation}\label{equ:propmap}
\Phi(d\Gamma)=d\Phi(\Gamma)
\end{equation}
for every $\Gamma\in\HGCS_{n}$.
\end{prop}
\begin{proof}

We have already checked that the degree of $\Phi$ is zero.
The proof of the other claim is similar to the proofs of \cite[Proposition 4.4]{MultiSourced} and \cite[Proposition 3.4]{AssarMarko}. Nevertheless, let us quickly go through the argument.

Let $\Gamma\in\bar\mV_v\bar\mE_e\bar\mH_S\G_{n}$.
It then holds that
$$
\Phi(d\Gamma)=
\Phi\left(\sum_{a\in E(\Gamma)}\Gamma/a\right)=
\sum_{a\in E(\Gamma)}\Phi\left(\Gamma/a\right)=
\sum_{a\in E(\Gamma)}\sum_{\tau\in F(\Gamma/a)}\Phi_{\tau}(\Gamma/a)
$$
where $\Gamma/a$ is the graph obtained by contracting the edge $a$ in $\Gamma$.
Spanning forests of $\Gamma/a$ are in natural bijection with spanning forests of $\Gamma$ that contain $a$, $\tau/a\leftrightarrow\tau$, so we can write
$$
\Phi(d\Gamma)=
\sum_{a\in E(\Gamma)}
\sum_{\substack{\tau\in F(\Gamma)\\ a\in E(\tau)}}\Phi_{\tau/a}(\Gamma/a)=
\sum_{\tau\in F(\Gamma)}\sum_{a\in E(\tau)}\Phi_{\tau/a}(\Gamma/a).
$$

\begin{lemma}
Let $\Gamma\in\bar\mV_v\bar\mE_e\bar\mH_S\G_{n}$, $\tau\in F(\Gamma)$ and $a\in E(\tau)$. Then
\begin{equation}
\Phi_{\tau/a}(\Gamma/a)\sim \Phi_{\tau}(\Gamma)/a,
\end{equation}
where $\sim$ means that they are in the same class of coinvariants under the action of $\sym_v\times\sym_e$.
\end{lemma}
\begin{proof}
It is clear that one side is $\pm$ the other side. Careful calculation of the sign is left to the reader.
\end{proof}

The lemma implies that
\begin{equation}
\Phi(d\Gamma)\sim\sum_{\tau\in F(\Gamma)}\sum_{a\in E(\tau)}\Phi_{\tau}(\Gamma)/a.
\end{equation}

Next we consider the right-hand side of \eqref{equ:propmap}, i.e., $d\Phi(\Gamma)$. There the differential decomposes as $d=d_C\pm d_E$, in a component $d_C$ that contracts ("core") internal directed edges and $d_E$ that replaces crossed edges by directed edges, see \eqref{equ:dEdef}.
Note also that the edges of $\Phi_\tau(\Gamma)$ are in 1-1-correspondence to those in $\Gamma$.
Given the spanning forest $\tau$, edges of $\Phi_\tau(\Gamma)$, i.e.\ of $\Gamma$, effected by the differential $d=d_C\pm d_E$ can be partitioned as
$$
E(\tau)\sqcup ED(\tau)\sqcup EC(\tau),
$$
where $E(\tau)$ is the set of edges between internal vertices in the forest $\tau$ and we define:


\begin{itemize}
\item $ED(\tau)$ is the set of edges that connect two connected components of $\tau$, including those attached to an external vertex if that vertex alone forms a connected component of $\tau$;
\item $EC(\tau)$ is the set of edges that make a cycle in a connected component of $\tau$.
\end{itemize}
Edges from $E(\tau)$ are effected by core differential $d_C$, and edges from $ED(\tau)$ and $EC(\tau)$ are effected by edge differential $d_E$.

Note that edges adjacent to external vertices can not be contracted, so they are not included in $E(\tau)$ if they are in the forest. But if they are not in the forest they are included in $ED(\tau)$ because edge differential can act on them.

Edges from $E(\tau)$ can be contracted by $d_C$, so
\begin{equation}
\label{eq:map1}
d_C \Phi(\Gamma)=
d_C\left(\sum_{\tau\in F(\Gamma)}\Phi_{\tau}(\Gamma)\right)=
\sum_{\tau\in F(\Gamma)}d_C\left(\Phi_{\tau}(\Gamma)\right)=
\sum_{\tau\in F(\Gamma)}\sum_{a\in E(\tau)}\Phi_{\tau}(\Gamma)/a\sim
\Phi(d\Gamma).
\end{equation}

The edge differential $d_E$ acts on edges of type $\Ess$, which are those in the sets $ED(\tau)$ and $EC(\tau)$. We then split
\begin{equation}
d_E(\Phi_\tau(\Gamma))=d_{ED}(\Phi_\tau(\Gamma))+d_{EC}(\Phi_\tau(\Gamma)),
\end{equation}
where
\begin{equation}
d_{ED}(\Phi_\tau(\Gamma))=\sum_{a\in ED(\tau)}d_E^{(a)}(\Phi_\tau(\Gamma)),\quad
d_{EC}(\Phi_\tau(\Gamma))=\sum_{a\in EC(\tau)}d_E^{(a)}(\Phi_\tau(\Gamma)),
\end{equation}

where $d_E^{(a)}$ maps edge $a$ as $\Ess\mapsto=\Ed-(-1)^{n+1}\dE$.

\begin{lemma}
\label{lem:map3}
Let $\Gamma\in\bar\mV_v\bar\mE_e\bar\mH_S\G_{n}$. Then
$$
\sum_{\tau\in F(\Gamma)}d_{EC}\left(\Phi_{\tau}(\Gamma)\right)\sim 0.
$$
\end{lemma}
\begin{proof}
Let
$$
N(\Gamma):=
\sum_{\tau\in F(\Gamma)}d_{EC}\left(\Phi_{\tau}(\Gamma)\right)=
\sum_{\tau\in F(\Gamma)}\sum_{a\in EC(\tau)}d_E^{(a)}\left(\Phi_{\tau}(\Gamma)\right).
$$

Terms in the above relation can be summed in another order. Let $FC(\Gamma)$ (\emph{cycled forests}) be the set of all sub-graphs $\rho$ of $\Gamma$ that contain all internal and external vertices, have $v+1$ edges, and whose every connected component has exactly one hair. Those graphs are similar to spanning forests, but have one cycle.
Let $C(\rho)$ be the set of edges in the cycle of $\rho$. Clearly, $\rho\setminus \{a\}$ for $a\in C(\rho)$ is a spanning forest of $\Gamma$ and sets $\{(\tau,a)|\tau\in F(\Gamma),a\in EC(\tau)\}$ and $\{(\rho,a)|\rho\in FC(\Gamma),a\in C(\rho)\}$ are bijective, so
$$
N(\Gamma)=\sum_{\rho\in FC(\Gamma)}\sum_{a\in C(\rho)}d_E^{(a)}\left(\Phi_{\rho\setminus \{a\}}(\Gamma)\right).
$$

It is now enough to show that
$$
\sum_{a\in C(\rho)}d_E^{(a)}\left(\Phi_{\rho\setminus \{a\}}(\Gamma)\right)\sim 0
$$
for every $\rho\in FC(\Gamma)$.
Let $y\in V(\Gamma)$ be the internal vertex in the cycle of $\rho$ closest to the external vertex of its connected component (along $\rho$). After choosing $a\in C(\rho)$, the cycle in $\Phi_{\rho\setminus \{a\}}(\Gamma)$ has the edge $a$ of type $\Ess$, and other edges of type $\Ed$ or $\dE$ with direction from $y$ to the edge $a$, such as in the following diagram.
$$
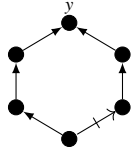
\begin{tikzpicture}[baseline=0ex,scale=.7]
 \node[int] (a) at (0,1.1) {};
 \node[int] (b) at (1,.5) {};
 \node[int] (c) at (1,-.5) {};
 \node[int] (d) at (0,-1.1) {};
 \node[int] (e) at (-1,-.5) {};
 \node[int] (f) at (-1,.5) {};
 \node[above] at (a) {$\scriptstyle y$};
 \draw (a) edge[latex-] (b);
 \draw (b) edge[latex-] (c);
 \draw (a) edge[latex-] (f);
 \draw (f) edge[latex-] (e);
 \draw (e) edge[latex-] (d);
 \draw (c) edge[<-,crossed] (d);
\end{tikzpicture}
$$
After acting by $d_E^{(a)}$ this $\Ess$ is replaced by $\Ed+(-1)^n\dE$, such as in the following diagram.
$$
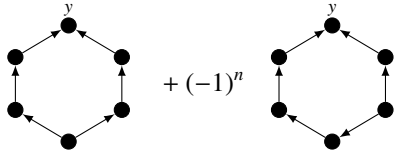
\begin{tikzpicture}[baseline=-.6ex,scale=.7]
 \node[int] (a) at (0,1.1) {};
 \node[int] (b) at (1,.5) {};
 \node[int] (c) at (1,-.5) {};
 \node[int] (d) at (0,-1.1) {};
 \node[int] (e) at (-1,-.5) {};
 \node[int] (f) at (-1,.5) {};
 \node[above] at (a) {$\scriptstyle y$};
 \draw (a) edge[latex-] (b);
 \draw (b) edge[latex-] (c);
 \draw (a) edge[latex-] (f);
 \draw (f) edge[latex-] (e);
 \draw (e) edge[latex-] (d);
 \draw (c) edge[latex-] (d);
\end{tikzpicture}
\quad+(-1)^n\quad
\begin{tikzpicture}[baseline=-.6ex,scale=.7]
 \node[int] (a) at (0,1.1) {};
 \node[int] (b) at (1,.5) {};
 \node[int] (c) at (1,-.5) {};
 \node[int] (d) at (0,-1.1) {};
 \node[int] (e) at (-1,-.5) {};
 \node[int] (f) at (-1,.5) {};
 \node[above] at (a) {$\scriptstyle y$};
 \draw (a) edge[latex-] (b);
 \draw (b) edge[latex-] (c);
 \draw (a) edge[latex-] (f);
 \draw (f) edge[latex-] (e);
 \draw (e) edge[latex-] (d);
 \draw (c) edge[-latex] (d);
\end{tikzpicture}
$$
Careful calculation of the sign shows that those two terms are cancelled with terms given from choosing neighboring edges in $C(\rho)$, and two last terms which do not have a corresponding neighbor are indeed $0$ as they have a type-cycle.
This concludes the proof that $N(\Gamma)\sim 0$.
\end{proof}

The similar study of the action on edges from $ED(\tau)$ leads to the following lemma.

\begin{lemma}
\label{lem:map4}
Let $\Gamma\in\bar\mV_v\bar\mE_e\bar\mH_S\G_{n}$. Then
$$
\sum_{\tau\in F(\Gamma)}d_{ED}\left(\Phi_{\tau}(\Gamma)\right)\sim 0.
$$
\end{lemma}
\begin{proof}
It holds that
$$
\sum_{\tau\in F(\Gamma)}d_{ED}\left(\Phi_{\tau}(\Gamma)\right)=
\sum_{\tau\in F(\Gamma)}\sum_{a\in ED(\tau)}d_E^{(a)}(\Phi_\tau(\Gamma)).
$$

Let $FD(\Gamma)$ (\emph{double-hair forests}) be the set of all sub-graphs $\lambda$ of $\Gamma$ that contain all internal and external vertices, have no cycles, whose one connected component has exactly two external vertices and whose other connected components have exactly one external vertex. Let those two external vertices be $j(\lambda),k(\lambda)\in S$.

For $\lambda\in FD(\Gamma)$ let $P(\lambda)$ be the set of edges in the path from $j(\lambda)$ to $k(\lambda)$. Clearly, $\lambda\setminus \{a\}$ for $a\in P(\lambda)$ is a spanning forest of $\Gamma$ and $a$ is in $ED(\Gamma)$ for that spanning forest. One can easily see that sets $\{(\tau,a)|\tau\in F(\Gamma),a\in ED(\tau)\}$ and $\{(\lambda,a)|\lambda\in FD(\Gamma),a\in P(\lambda)\}$ are bijective, so
$$
\sum_{\tau\in F(\Gamma)}d_{ED}\left(\Phi_{\tau}(\Gamma)\right)=
\sum_{\lambda\in FD(\Gamma)}\sum_{a\in P(\lambda)}d_E^{(a)}\left(\Phi_{\lambda\setminus\{a\}}(\Gamma)\right).
$$

To finish the proof it is enough to show that
$$
\sum_{a\in P(\lambda)}d_E^{(a)}\left(\Phi_{\lambda\setminus\{a\}}(\Gamma)\right)\sim
0
$$
for every $\lambda\in FD(\Gamma)$. After choosing $a\in P(\lambda)$ the path from $j(\lambda)$ to $k(\lambda)$ along $\lambda$ in $\Phi_{\lambda\setminus\{a\}}(\Gamma)$ has the edge $a$ of type $\Ess$, and the other edges of type $\Ed$ or $\dE$ with direction from $j(\lambda)$ or $k(\lambda)$ to the edge $a$, such as in the following diagram.
$$
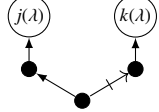
\begin{tikzpicture}[baseline=0ex,scale=.7]
 \node[ext] (b) at (1,.5) {$\scriptstyle k(\lambda)$};
 \node[int] (c) at (1,-.5) {};
 \node[int] (d) at (0,-1.1) {};
 \node[int] (e) at (-1,-.5) {};
 \node[ext] (f) at (-1,.5) {$\scriptstyle j(\lambda)$};
 \draw (b) edge[latex-] (c);
 \draw (f) edge[latex-] (e);
 \draw (e) edge[latex-] (d);
 \draw (c) edge[<-,crossed] (d);
\end{tikzpicture}
$$
After acting by $d_E^{(a)}$ this $\Ess$ is replaced by $\Ed+(-1)^n\dE$, such as in the following diagram.
$$
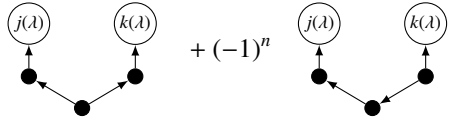
\begin{tikzpicture}[baseline=-.6ex,scale=.7]
 \node[ext] (b) at (1,.5) {$\scriptstyle k(\lambda)$};
 \node[int] (c) at (1,-.5) {};
 \node[int] (d) at (0,-1.1) {};
 \node[int] (e) at (-1,-.5) {};
 \node[ext] (f) at (-1,.5) {$\scriptstyle j(\lambda)$};
 \draw (b) edge[latex-] (c);
 \draw (f) edge[latex-] (e);
 \draw (e) edge[latex-] (d);
 \draw (c) edge[latex-] (d);
\end{tikzpicture}
\quad+(-1)^n\quad
\begin{tikzpicture}[baseline=-.6ex,scale=.7]
 \node[ext] (b) at (1,.5) {$\scriptstyle k(\lambda)$};
 \node[int] (c) at (1,-.5) {};
 \node[int] (d) at (0,-1.1) {};
 \node[int] (e) at (-1,-.5) {};
 \node[ext] (f) at (-1,.5) {$\scriptstyle j(\lambda)$};
 \draw (b) edge[latex-] (c);
 \draw (f) edge[latex-] (e);
 \draw (e) edge[latex-] (d);
 \draw (c) edge[-latex] (d);
\end{tikzpicture}
$$
Careful calculation of the sign shows that those two terms are cancelled with terms given from choosing neighboring edges in $P(\lambda)$. The
two last terms which does not have corresponding neighbour are zero because they have external vertex which is not target.
\end{proof}

Equation \eqref{eq:map1}, and Lemmas \ref{lem:map3} and \ref{lem:map4} imply that
$$
\Phi(d(\Gamma))\sim
d_C(\Phi(\Gamma))
\sim d_C(\Phi(\Gamma))+
\sum_{\tau\in F(\Gamma)}d_{EC}\left(\Phi_{\tau}(\Gamma)\right)+
\sum_{\tau\in F(\Gamma)}d_{ED}\left(\Phi_{\tau}(\Gamma)\right)=
d_C(\Phi(\Gamma))+d_E(\Phi(\Gamma))=d(\Phi(\Gamma)).
$$

After taking coinvariants this implies that
$$
\Phi(d(\Gamma))+\Phi(\chi(\Gamma))=d(\Phi(\Gamma))
$$
for each $\Gamma\in\HGS_{n}$. Hence, $\Phi:(\HGS_{n},d)\to (\HOGS_{n+1},d)$ is a map of complexes.
\end{proof}

After (anti-)symmetrizing external vertices this map induces the map
$$
\Phi: \left(\HG_{m,n},d\right)\to\left(\HOG_{m,n+1},d\right).
$$

\subsubsection{The dual map}\label{sec:thedualmap}
The dual of $\Phi$ is
$$
\F: \left(\HOGCS_{n+1},\delta\right) \to \left(\HGCS_{n},\delta\right),
$$
and the dual of its version for (anti-)symmetrized external vertices is
$$
\F: \left(\HOGC_{m,n},\delta\right)\to\left(\HGC_{m,n+1},\delta\right).
$$
Let us describe these maps combinatorially, to see that they are relatively simple and straightforward to compute.

\begin{defi}
Let $\Gamma\in\HOGC_{m,n}$ be a single term graph. We call it a \emph{forest graph} if all its internal vertices that are at least 3-valent have exactly 1 outgoing edge.
\end{defi}

\begin{lemma}
\label{lem:fg}
Let $\Gamma\in\HOGC_{m,n}$ be a (single term) graph that is not a forest graph. Then $\F(\Gamma)=0$.
\end{lemma}
\begin{proof}
In the dual picture one easily checks that every graph in the linear combination $\Phi(\gamma)$ for $\gamma\in\HG_{m,n+1}$ is a forest graph. This leads to the result.
\end{proof}

If $\Gamma\in\HOGC_{m,n}$ is a forest graph, $\F(\Gamma)$ is a graph in $\HGC_{m,n+1}$ up to the sign obtained from $\Gamma$ by replacing each occurrence of $\dEd$ with a single edge and ignoring the edge directions.
Here are some examples:
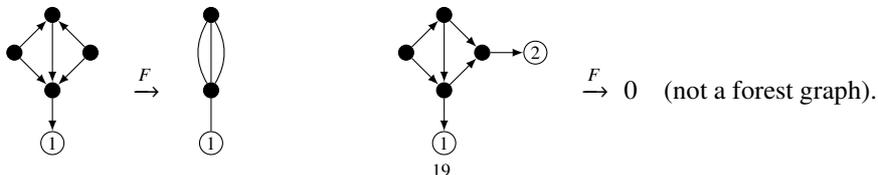
\begin{align*}
\begin{tikzpicture}[scale=1]
	\node[int] (v1) at (0,0) {};
	\node[int] (v2) at (0,1) {};
	\node[int] (v3) at (-.5,0.5) {};
	\node[int] (v4) at (.5,0.5) {};
	\node[ext] (w) at (0,-0.7) {$\scriptstyle 1$};
\draw[-latex] (v3) edge (v2) edge (v1) (v4) edge (v2) edge (v1) (v1) edge (w) (v2) edge (v1);
\end{tikzpicture}
&
\quad\xrightarrow{F} \quad
\begin{tikzpicture}[scale=1]
	\node[int] (v1) at (0,0) {};
	\node[int] (v2) at (0,1) {};
	\node[ext] (w) at (0,-0.7) {$\scriptstyle 1$};
\draw (v1) edge (v2) edge[bend left] (v2) edge[bend right] (v2) edge (w) ;
\end{tikzpicture}
&
\begin{tikzpicture}[scale=1]
	\node[int] (v1) at (0,0) {};
	\node[int] (v2) at (0,1) {};
	\node[int] (v3) at (-.5,0.5) {};
	\node[int] (v4) at (.5,0.5) {};
	\node[ext] (w) at (0,-0.7) {$\scriptstyle 1$};
	\node[ext] (w2) at (1.2,0.5) {$\scriptstyle 2$};
\draw[-latex] (v3) edge (v2) edge (v1) (v2) edge (v4) (v1) edge (v4) 
(v1) edge (w) (v2) edge (v1) (v4) edge (w2);
\end{tikzpicture}
\quad\xrightarrow{F}\;
0 \quad \text{(not a forest graph).}
\end{align*}

\subsection{The proof of Theorem  \ref{thm:main}} 

In this subsection we prove Theorem \ref{thm:main} by proving its dual version:

\begin{prop}
\label{prop:main1}
The map 
$$
\Phi:  \left(\HGS_n,d\right)\to\left(\HOGS_{n+1},d\right)
$$
is a quasi-isomorphism.
\end{prop}
\begin{proof}
Using Proposition \ref{prop:sk} it is enough to prove that
$$
\Phi:  \left(\HGS_n,d\right)\to\left(\HOsGS_{n+1},d\right)
$$
is a quasi-isomorphism.

On its mapping cone we set up the spectral sequence on the number of vertices. Our complexes split into finite dimensional subcomplexes according to the loop number, hence the spectral sequence converges. It is therefore enough to prove the claim for the first differential of the spectral sequence.

Since the pieces of the differentials that contract edges lower the number of vertices, it is clear that on the first page of the spectral sequences there is the mapping cone of the map
$$
\Phi:  \left(\HGS_n,0\right)\to\left(\HOsGS_{n+1},d_E\right).
$$

These complexes are now direct sums of subcomplexes spanned by graphs with a fixed number of vertices and edges, so it is enough to show the claim for
$$
\Phi:  \left(\mV_v\mE_e\bar\mH_S\G_n,0\right)\to\left(\mV_v\mE_e\bar\mH_S\mO^{sk}\G_{n+1},d_E\right).
$$

Recall from \eqref{eq:VEHG} and \eqref{eq:VEHOsG} that both $\mV_v\mE_e\bar\mH_S\G_n$ and the skeleton complex $\mV_v\mE_e\bar\mH_S\mO^{sk}\G_{n+1}$ are spaces of invariants of the action of $\sym_v\times \left(\sym_e\ltimes \sym_2^{\times e}\right)$.
The action clearly commutes with the map $\Phi$. Since the edge differential $d_E$ does not change the number of vertices and edges, taking homology commutes with taking coinvariants of that action. Therefore, it is now enough to show the claim for
$$
\Phi:  \left(\bar\mV_v\bar\mE_e\bar\mH_S\G_n,0\right)\to\left(\bar\mV_v\bar\mE_e\bar\mH_S\mO^{sk}\G_{n+1},d_E\right).
$$

Let us pick up a particular (single term) graph $\Gamma\in\bar\mV_v\bar\mE_e\bar\mH_S\G_n$. Let $\langle\mO\Gamma\rangle$ be the subspace of $\bar\mV_v\bar\mE_e\bar\mH_S\mO^{sk}\G_{n+1}$ spanned by skeleton graphs with the core graph $\Gamma$.

The map $\Phi$ is defined such that $\Phi(\Gamma)\in\langle\mO\Gamma\rangle$. Also, differential $d_E$ acts within particular subspace $\langle\mO\Gamma\rangle$. Therefore, we can split the map as a direct sum and it is enough to prove the clam for
\begin{equation}
\label{eq:hreduced}
\Phi:  \left(\langle\Gamma\rangle,0\right)\to
\left(\langle\mO\Gamma\rangle,d_{E}\right),
\end{equation}
for every $\Gamma\in\bar\mV_v\bar\mE_e\bar\mH_S\G_n$.

In order to prove that, let us choose $v$ edges in $\Gamma$, say $a_1,\dots,a_{v}$. Let $F(a_1,\dots,a_i)$ be the sub-graph of $\Gamma$ that includes those edges, all external vertices and all necessary internal vertices. We require that for every $i=1,\dots,v$ the sub-graph $F(a_1,\dots,a_i)$ is a forest. Recall that in a forest, every connected component has exactly one external vertex. Clearly, $F(a_1,\dots, a_{v})$ is a spanning forest.

For every $i=0,\dots,v$, we form a graph complex $\langle\mO\Gamma^i\rangle$ as follows: it is spanned by graphs with a core graph $\Gamma$ with attached edge types from $\bar\sigma:=\{\Ed[n],\dE[n],\Ess[n-1],\ET[n]\}$
such that:
\begin{itemize}
\item edges $a_1,\dots, a_i$ have type $\ET$, and other edges have other types;
\item no (internal or external) vertex in the forest $F(a_1,\dots,a_i)$ has a neighbouring edge of type $\Ed$ or $\dE$ heading away from it;
\item every internal vertex outside the forest $F(a_1,\dots,a_i)$ has a neighbouring edge of type $\Ed$ or $\dE$ heading away from it (it is not a sink);
\item there are no cycles along arrows on edges of type $\Ed$ and $\dE$.
\end{itemize}

Examples of graphs in $\langle\mO\Gamma^i\rangle$ are shown in Figure \ref{fig:SGamma}.

\begin{figure}[H]
$$
\begin{tikzpicture}[scale=1.4]
 \node[red,int] (a) at (0,0) {};
 \node[red,int] (b) at (1,0) {};
 \node[int] (c) at (1,1) {};
 \node[int] (d) at (0,1) {};
 \node[red,int] (a1) at (-.5,-.5) {};
 \node[red,int] (b1) at (1.5,-.5) {};
 \node[int] (c1) at (1.5,1.5) {};
 \node[red,int] (d1) at (-.5,1.5) {};
 \node[red,int] (x) at (-1,.5) {};
 \node[int] (y) at (1.5,.5) {};
 \node[int] (z) at (2.2,.5) {};
 \node[red,ext] (e1) at (-1.6,.5) {\color{black}1};
 \node[red,ext] (e2) at (-.9,-.9) {\color{black}2};
 \node[red,ext] (e3) at (1.9,-.9) {\color{black}3};
 \node[red,ext] (e4) at (1.9,1.9) {\color{black}4};
 \draw (a) edge[red,very thick] (b);
 \draw (b) edge[crossed,->] (c);
 \draw (c) edge[crossed,->] (d);
 \draw (d) edge[-latex] (a);
 \draw (a1) edge[crossed,->] (b1);
 \draw (b1) edge[latex-] (y);
 \draw (y) edge[latex-] (c1);
 \draw (b1) edge[crossed,->] (z);
 \draw (z) edge[latex-] (c1);
 \draw (y) edge[latex-] (z);
 \draw (c1) edge[crossed,->] (d1);
 \draw (d1) edge[crossed,->] (a1);
 \draw (a) edge[red,very thick] (a1);
 \draw (b) edge[crossed,->] (b1);
 \draw (c) edge[-latex] (c1);
 \draw (d) edge[-latex] (d1);
 \draw (a1) edge[crossed,->] (x);
 \draw (x) edge[red,very thick] (d1);
  \draw (x) edge[red,very thick] (e1);
 \draw (a1) edge[red,very thick] (e2);
 \draw (b1) edge[red,very thick] (e3);
 \draw (c1) edge[-latex] (e4);
\end{tikzpicture}
$$
\caption{\label{fig:SGamma}
An example of graph in $\langle\mO\Gamma^i\rangle$, with $\Gamma$ as in Figure \ref{fig:span}. The forest $F(a_1,a_2,a_3,a_4,a_5,a_6)$ is depicted red.}
\end{figure}
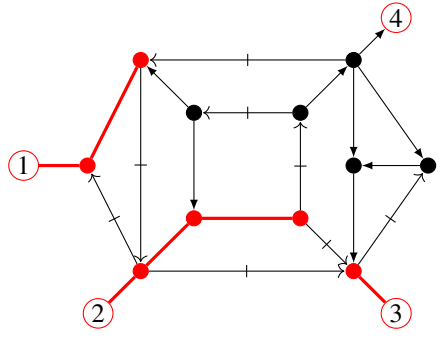

The differential on $\langle\mO\Gamma^i\rangle$ is the edge differential $d_E$ induced by
$$
\Ess\mapsto\Ed-(-1)^n\dE,
$$
as usual. If a resulting graph does not fulfil the conditions above, it is considered zero. Note that thick edges are not effected by the differential.

It is straightforward to check that
\begin{equation}
\left(\langle\mO\Gamma^0\rangle,d_E\right)=
\left(\langle\mO\Gamma\rangle,d_E\right).
\end{equation}
Also, it holds that $\langle\mO\Gamma^{v}\rangle $ is one dimensional, spanned by the graph with edges $a_1,\dots,a_{v}$ of type $\ET$ and other edges of type $\Ess$.

For every $i=1,\dots,v$, there is a map
\begin{equation}
f^i:\langle\mO\Gamma^{i-1}\rangle\rightarrow\langle\mO\Gamma^i\rangle
\end{equation}
that only change the type of the edge $a_i$ as
\begin{equation}
\Ess\mapsto 0,\quad
\Ed\mapsto\ET,\quad
\dE\mapsto(-1)^{n+1}\ET,
\end{equation}
where forbidden graphs are considered zero.

\begin{lemma}
For every $i=1,\dots,v$ the map $f^i:\langle\mO\Gamma^{i-1}\rangle\rightarrow\langle\mO\Gamma^i\rangle$ is a quasi-isomorphism.
\end{lemma}
\begin{proof}
The essential difference between $\langle\mO\Gamma^{i-1}\rangle$ and $\langle\mO\Gamma^{i}\rangle$ is in the edge $a_i$, it has to be of type $\ET$ in $\langle\mO\Gamma^{i}\rangle$, and it is of another type in $\langle\mO\Gamma^{i-1}\rangle$.
Since $f^i$ does not change types of other edges, it splits as a direct sum of maps between complexes with fixed types of other edges
$$
f^i_{fix}:\langle\mO\Gamma^{i-1}_{fix}\rangle\rightarrow\langle\mO\Gamma^i_{fix}\rangle
$$
where $\langle\mO\Gamma^{i-1}_{fix}\rangle$ and $\langle\mO\Gamma^{i}_{fix}\rangle$ are sub-complexes spanned by graphs with fixed types of all edges other than $a_i$. It is enough to show that each $f^i_{fix}$ is a quasi-isomorphism.

Here, depending on the choice of fixed edge types, the conditions of the complex can disallow some possibilities for the edge $a_i$ in both $\langle\mO\Gamma^{i-1}_{fix}\rangle$ and $\langle\mO\Gamma^{i}_{fix}\rangle$. We list all cases, showing that the map is a quasi-isomorphism in all of them. Let the vertex that is in the forest $F(a_1,\dots,a_{i})$ but not in the forest $F(a_1,\dots,a_{i-1})$ be called $x_i$.
\begin{enumerate}
\item If there is a vertex in the forest $F(a_1,\dots,a_{i-1})$ that has a neighbouring edge of type $\Ed$ or $\dE$ heading away from it, or there is a vertex outside the forest $F(a_1,\dots,a_{i})$ that does not have a neighbouring edge of type $\Ed$ or $\dE$ heading away from it, or there is a cycle along arrows on edges of type $\Ed$ or $\dE$ outside the forest $F(a_1,\dots,a_{i})$, both $\langle\mO\Gamma^{i-1}_{fix}\rangle$ and $\langle\mO\Gamma^{i}_{fix}\rangle$ are zero complexes and the map is clearly a quasi-isomorphism.
\item If the conditions of (1) do not hold and $x_i$ has a neighbouring edge of type $\Ed$ or $\dE$ heading away from it, the edge $a_i$ (that goes from a vertex in the forest $F(a_1,\dots,a_{i-1})$ towards $x_i$) can have types $\Ed$ or $\Ess$ in $\langle\mO\Gamma^{i-1}_{fix}\rangle$, making the complex acyclic. In $\langle\mO\Gamma^{i}_{fix}\rangle$, no type is allowed, so it is again the zero complex. Therefore, the map is again a quasi-isomorphism.
\item If the conditions of (1) do not hold and $x_i$ does not have a neighbouring edge of type $\Ed$ or $\dE$ heading away from it, the edge $a_i$ must have type $\Ed$ in $\langle\mO\Gamma^{i-1}_{fix}\rangle$. In $\langle\mO\Gamma^{i}_{fix}\rangle$ that edge must have type $\ET$, making the map an isomorphism. Thus, it is also a quasi-isomorphism.
\end{enumerate}
\end{proof}

The lemma implies that
\begin{equation}
f:=f^{v}\circ\dots\circ f^1:
\langle\mO\Gamma\rangle\rightarrow\langle\mO\Gamma^{v}\rangle
\end{equation}
is a quasi-isomorphism.

\begin{lemma}
\label{lem:QIfh}
The map $f\circ \Phi:\langle\Gamma\rangle\rightarrow\langle\mO\Gamma^{v}\rangle$ is a quasi-isomorphism.
\end{lemma}
\begin{proof}
Both complexes are one-dimensional, so we only need to check that $f\circ \Phi\neq 0$.
The left-hand side complex has a generator $\Gamma$. It holds that
$$
f\circ \Phi(\Gamma)=
f\left(\sum_{\tau\in F(\Gamma)}\Phi_{\tau}(\Gamma)\right).
$$
The map $\Phi_{\tau}$ gives edges in $E(\tau)$ type $\Ed$ or $\dE$, and type $\Ess$ to the other edges. After that, the map $f=f^1\circ\dots\circ f^{v}$ kills all graphs with any of edges $a_1,\dots,a_{v}$ being of type $\Ess$. Therefore, $f\circ h_{x,\tau}$ is non-zero only if the forest $\tau$ consist exactly of the edges $a_1,\dots,a_{v}$. Let us call this forest $T$. So
\begin{equation}
f\circ \Phi(\Gamma)=
f\left(\Phi_{T}(\Gamma)\right).
\end{equation}
It is clearly the generator of $\langle\mO\Gamma^{v}\rangle$, and therefore non-zero.
\end{proof}

Since $f$ and $f\circ \Phi$ are quasi-isomorphism, it follows that $\Phi$ is also a quasi isomorphism, what was to be demonstrated.
\end{proof}

The following corollary is now straightforward.

\begin{cor}
\label{cor:main1}
The induced map 
$$
\Phi:  \left(\HG_{m,n},d\right)\to\left(\HOG_{m,n+1},d\right)
$$
is a quasi-isomorphism.
\end{cor}

Using Proposition \ref{prop:incl} we get the following corollary. It has already been shown as part of \cite[Theorem 1.1.]{AssarMarko}.

\begin{cor}
There is an explicit quasi-isomorphism
$$
\left(\HG_{n,n},d\right)\to\left(\OG_{n+1},d_0\right).
$$
\end{cor}

\subsection{The proof of Theorem \ref{thm:mainlie}}

Corollary \ref{cor:main1} implies that also the dual map
$$
\F:  \left(\HOGC_{m,n+1},\delta\right)\to\left(\HGC_{m,n},\delta\right)
$$
is a quasi-isomorphism. On this complexes we have Lie algebra structures.

\begin{prop}
The map $\F:  \left(\HOGC_{m,n+1},\delta\right)\to\left(\HGC_{m,n},\delta\right)$ respects the Lie algebra structures, i.e.\
$$
\F([\Gamma,\Gamma'])=[\F(\Gamma),\F(\Gamma')]
$$
for every $\Gamma,\Gamma'\in \HOGC_{m,n+1}$.
\end{prop}
\begin{proof}
It is enough the check the relation for single term graphs $\Gamma$ and $\Gamma'$. Recall from Lemma \ref{lem:fg} that $F(\Gamma)=0$ unless $\Gamma$ is a forest graph.
It is easy to see that if either of $\Gamma$ or $\Gamma'$ is not a forest graph, neither is its Lie bracket $[\Gamma,\Gamma']$. So it is enough to check the relation for forest graphs $\Gamma$ and $\Gamma'$.

In constructing $[\Gamma,\Gamma']$ a hair from one graph can connect to any vertex from another graph. But if the hair is connected to a 2-valent vertex of the form $\dEd$ that comes from a skeleton edge $\Ess$, the resulting graph is not a forest graph, so it is sent to zero after acting by $\F$. Therefore, to prove the relation, we need to consider only cases where a hair is connected to at least 3-valent vertices. They come from skeleton vertices.

It is now clear that connecting hairs before and after the action of $\F$ yields the same result. Careful calculation of the sign is left to the reader.
\end{proof}

For the next assertion of Theorem \ref{thm:mainlie} we need to check that $\F(m_1)=m_0$. The only term in the series of graphs
\[
m_1 =
\sum_{k\geq 2} 
\frac 2 {k!}
\underbrace{
\begin{tikzpicture}[baseline=-.65ex]
\node[int] (v) at (0,.5) {};
\node[ext] (h1) at (-.7,-.2) {\;};
\node[ext] (h2) at (-.3,-.2) {\;};
\node[ext] (h3) at +(.7,-.2) {\;};
\draw (v) edge[-latex] (h1) edge[-latex] (h2)  edge[-latex] (h3) +(.25,-.7) node {$\scriptstyle \dots$};
\end{tikzpicture}
}_{k\times}
\in \HOGC_{n,n+1},
\]
that is a forest graph is $
\begin{tikzpicture}[baseline=-.65ex,scale=.5]
 \node[ext] (a) at (0,0) {\;};
 \node[int] (b) at (1,0) {};
 \node[ext] (c) at (2,0) {\;};
 \draw (a) edge[latex-] (b);
 \draw (b) edge[-latex] (c);
\end{tikzpicture}
$,
and it is sent to $m_0=
\begin{tikzpicture}[baseline=-.65ex,scale=.5]
 \node[ext] (a) at (0,0) {\;};
 \node[ext] (c) at (1,0) {\;};
 \draw (a) edge (c);
\end{tikzpicture}
$.

Finally we need to check that the diagram of complexes \eqref{equ:thmmainliecd} homotopy commutes.
We will do this by considering a one sided inverse $\phi$ to the lower horizontal arrow, that has first been introduced in \cite{TW2}.
\[
	\begin{tikzcd}
		\OGC_{n+1}[-1] \ar{r}{\simeq} \ar{d}{\simeq_{\geq 2}} & (\HOGC_{n,n+1}, \delta+[m_1,-]) \ar{d}{\simeq}\\
	\GC_n[-1] \ar{r}{\simeq_{\geq 2}} &  (\HGC_{n,n},\delta+[m_0,-]) \ar[bend right]{l}{\phi}
	\end{tikzcd}	
\]
Concretely, for a graph $\Gamma\in \HGC_{n,n}$ of loop order $g$ we set 
\[
\phi(\Gamma) = 
\begin{cases}
0 & \text{if $\Gamma$ has $\geq 2$ hairs} \\
0 & \text{if $\Gamma$ has a single hair connected to a trivalent vertex} \\
\pm \frac{(\# neighbors)-2}{2g} (\Gamma- \text{hair}) & \text{if $\Gamma$ has a single hair}
\end{cases}\, ,
\]
where $(\Gamma- \text{hair})$ is the graph $\Gamma$ with the single hair removed and $(\# neighbors)$ is the number of neighbors of the vertex the hair connects to (not counting the hair as a neighbor).
The verification that this is indeed a one sided inverse to the map $\GC_n[-1] \to  \HGC_{n,n}$ we refer to \cite{TW2}, or leave it to the reader as an exercise.
Finally, one just has to note that the inner square in the above diagram commutes in loop orders $\geq 2$, using the description of the map $\OGC_{n+1}\to \GC_n$ of section \ref{sec:GCOGCcomparison}. Mind that the factor $\frac{1}{2g}$ appearing in the definition of $\phi$ is the reason for us introducing a similar (conventional) factor in \eqref{equ:GOGmapdef}, that was absent in \cite{MultiOriented,MultiSourced}
This finishes the proof of Theorem \ref{thm:mainlie}. \hfill \qed

\section{The map from the oriented to the ribbon graph complex}\label{sec:HOGCtoRGC}

In this section we shall discuss the connection of the graph complexes of the previous sections to the ribbon graph complex, introduced by Penner and Kontsevich.
We will start by recalling some definitions and constructions from \cite{MW}. 

\subsection{Recollections from \cite{MW}}

The main player in \cite{MW} is the ribbon graph properad $\RGra_0$. The space of operations $\RGra(r,s)$ with $s$ inputs and $r$ outputs is the space of linear combinations of connected ribbon graphs with the set of vertices identified with $[r]$ and the set of boundary components identified with $[s]$. A ribbon graph (or fat graph) is a graph with a prescripition of a cyclic ordering of the incident half-edges at each vertex. Thickening the graph, one obtains an oriented surface with some disks removed.
Here is an example ribbon graph in $\RGra_0(3,2)$:
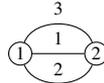
\[
\begin{tikzpicture}
\node[ext] (v1) at (0,0) {$\scriptstyle 1$};
\node[ext] (v2) at (1,0) {$\scriptstyle 2$};
\draw (v1) edge (v2) edge[bend left=70] (v2) edge [bend right=70] (v2);
\node  at (0.5,0.2) {$\scriptstyle 1$};
\node  at (0.5,-0.2) {$\scriptstyle 2$};
\node  at (0.5,0.6) {$\scriptstyle 3$};
\end{tikzpicture}
\]
Here one should think of the edges being thickened to ribbons, which connect at the vertices in the indicated cyclic order.
 The properadic compositions are obtained by "connecting a vertex and a boundary component" in the sense that a vertex is deleted, and the incident edges are distributed along the boundary component in all planar possible ways.
We refer to \cite[section 4]{MW} for more combinatorial details, and also for the precise sign and degree conventions.
We just remark that a ribbon graph $\Gamma\in \RGra_0(r,s)$ with $k$ edges has cohomological degree $+k$ by covention, i.e., formally each edge has degree $+1$.

The important fact is that one has a map of properads
\[
\La\LieB \to \RGra_0,
\]
from the (degree shifted) Lie bialgebra properad, defined on the bracket and cobracket generator by the following formulas:
\begin{equation}\label{equ:LieBmapdef}
\begin{aligned}
\text{bracket:}
\begin{tikzpicture}
\node[ext] (v) at (0,0) {};
\draw (v) edge +(0,.4) edge +(-.3,-.3) edge (.3,-.3);
\end{tikzpicture}
&\mapsto
\begin{tikzpicture}
\node[ext] (v) at (0,0) {$\scriptstyle 1$};
\node[ext] (w) at (.5,0) {$\scriptstyle 2$};
\node at (.25,.3) {$\scriptstyle 1$};
\draw (v) edge (w);
\end{tikzpicture}\in \RGra_0(1,2) \\
\text{cobracket:}
\begin{tikzpicture}[yscale=-1];
\node[ext] (v) at (0,0) {};
\draw (v) edge +(0,.4) edge +(-.3,-.3) edge (.3,-.3);
\end{tikzpicture}
&\mapsto
\begin{tikzpicture}[every loop/.style={draw,-, looseness=10}]
\node[ext] (v) at (0,0) {$\scriptstyle 1$};
\node at (0,.3) {$\scriptstyle 1$};
\node at (.5,.3) {$\scriptstyle 2$};
\draw (v) edge[in=40, out=140, loop] (v);
\end{tikzpicture}\in \RGra_0(2,1)
\end{aligned}
\end{equation}
Mind that the bracket and cobracket generator both have degree $1$ and are symmetric (co)operations, due to the degree shift compared to the standard Lie bialgebra properad.
We will also define a different properad map we denote by $*:\La\LieB\to \RGra_0$, simply by sending the cobracket generator to zero,
\begin{equation}\label{equ:LieBstarmapdef}
\begin{aligned}
\ILieB &\xrightarrow{*} \RGra_0 \\
\text{bracket:}
\begin{tikzpicture}
\node[ext] (v) at (0,0) {};
\draw (v) edge +(0,.4) edge +(-.3,-.3) edge (.3,-.3);
\end{tikzpicture}
&\mapsto
\begin{tikzpicture}
\node[ext] (v) at (0,0) {$\scriptstyle 1$};
\node[ext] (w) at (.5,0) {$\scriptstyle 2$};
\node at (.25,.3) {$\scriptstyle 1$};
\draw (v) edge (w);
\end{tikzpicture}\in \RGra_0(1,2) \\
\text{cobracket:}
\begin{tikzpicture}[yscale=-1];
\node[ext] (v) at (0,0) {};
\draw (v) edge +(0,.4) edge +(-.3,-.3) edge (.3,-.3);
\end{tikzpicture}
&\mapsto 0\in \RGra_0(2,1)
\end{aligned}
\end{equation}
Furthermore, we may compose each of the above two maps with the projection $\La\LieB_\infty\to \La\LieB$.
Then one can consider the properadic deformation complex 
\[
 \Def(\La\LieB_\infty\xrightarrow{*} \RGra_0)
\]
(see \cite{MW, MWDef} for the definition and conventions), whose elements are essentially series of ribbon graphs with un-labelled vertices and boundary components.
The subcomplex 
\[
 \RGC_0 \subset \Def(\La\LieB_\infty\xrightarrow{*} \RGra_0)
\]
consisting of graphs with all vertices of valence $\geq 3$ is the Kontsevich-Penner ribbon graph complex, computing the compactly supported cohomology of the moduli spaces of curves, with antisymmetrized punctures.\footnote{In fact, the trivalence condition is not very important, and $H(\Def(\La\LieB_\infty \xrightarrow{*} \RGra_0))$ is only slightly larger than $H(\RGC_0)$.}
\[
H(\RGC_0) \cong \prod_{g\geq 0, n\geq 1 \atop 2g+n\geq 3}(H_c(\MM_{g,n})\otimes \Q[-1]^{\otimes n})_{\bbS_n}
\]
Combinatorially, elements of $\RGC_0$ can be seen as series of ribbon graphs, with unidentifiable vertices and boundary components of degree 0 and edges of degree 1.

The above properadic definition of the ribbon graph complex as a deformation complex has three interesting consequences explored in \cite{MW}:
\begin{enumerate}
	\item Since deformation complexes are dg Lie algebras, one finds that $\RGC_0$ carries a dg Lie structure.
	\item Instead of deforming the map $*$ of \eqref{equ:LieBstarmapdef} we can as well deform the map \eqref{equ:LieBmapdef} and consider the complex 
	\[
	\Def(\La\LieB_\infty \xrightarrow{} \RGra_0).
	\]  
	The former complex in particular is a deformation of the deformation complex of the map $*$, and from this one in particular obtains an additional differential $\Delta$ on the ribbon graph complex $\RGC_0$.
	Combinatorially, one can check that this differential acts on a ribbon graph by spitting a boundary component in two by adding one edge across the component, in all possible ways, schematically:
	\begin{equation}\label{equ:bridgelanddiff_illustration}
	\Delta\colon\,
\begin{tikzpicture}[scale=.5]
\node [int] (v4) at (1.5,0) {};
\node [int] (v5) at (0.46,1.43) {};
\node [int] (v2) at (-1.21,0.88) {};
\node [int] (v1) at (-1.21,-0.88) {};
\node [int] (v3) at (0.46,-1.43) {};
\draw  (v2) edge (v1);
\draw  (v1) edge (v3);
\draw  (v3) edge (v4);
\draw  (v4) edge (v5);
\draw  (v5) edge (v2);
\node  (v6) at (-2.02,1.47) {};
\node  (v7) at (0.77,2.38) {};
\node  (v8) at (2.5,0) {};
\node  (v9) at (0.77,-2.38) {};
\node  (v10) at (-2.02,-1.47) {};
\draw  (v2) edge (v6);
\draw  (v5) edge (v7);
\draw  (v4) edge (v8);
\draw  (v3) edge (v9);
\draw  (v1) edge (v10);
\end{tikzpicture}
\mapsto 
\sum \,
\begin{tikzpicture}[scale=.5]
\node [int] (v4) at (1.5,0) {};
\node [int] (v5) at (0.46,1.43) {};
\node [int] (v2) at (-1.21,0.88) {};
\node [int] (v1) at (-1.21,-0.88) {};
\node [int] (v3) at (0.46,-1.43) {};
\draw (v1) edge (v4);
\draw  (v2) edge (v1);
\draw  (v1) edge (v3);
\draw  (v3) edge (v4);
\draw  (v4) edge (v5);
\draw  (v5) edge (v2);
\node  (v6) at (-2.02,1.47) {};
\node  (v7) at (0.77,2.38) {};
\node  (v8) at (2.5,0) {};
\node  (v9) at (0.77,-2.38) {};
\node  (v10) at (-2.02,-1.47) {};
\draw  (v2) edge (v6);
\draw  (v5) edge (v7);
\draw  (v4) edge (v8);
\draw  (v3) edge (v9);
\draw  (v1) edge (v10);
\end{tikzpicture}
	\end{equation}
	This operation has been first considered by T. Bridgeland to our knowledge. A. C\v ald\v araru's Conjecture (see Conjecture \ref{conj:caldararu} below) states that the complex $(\RGC_0, \delta+\Delta)$ computes the compactly supported cohomology of the moduli spaces without maked points $\MM_{g}$.

	\item Since derivations of the properad $\La\LieB_\infty$ obviously map into the deformation complex above, and since those derivations can (homotopically) be identified with the graph complex $\HOGC_{0,1}$ we obtain the following maps of complexes
	\begin{equation}\label{equ:Def RGC GC}
	\begin{tikzcd}
	& \Def(\La\LieB_\infty\xrightarrow{\mathit{id}}\La\LieB_\infty) \ar{r} & \Def(\LieB_\infty \xrightarrow{} \RGra) \\
	\GC_{0}[-1] & \OGC_{1}[-1] \ar{l}{\simeq_{\geq 2}} \ar{r} \ar[hookrightarrow]{u}{\simeq_{\geq 2}} & (\RGC_0, \delta+\Delta) \ar[hookrightarrow]{u}{\simeq_{\geq 2}}
	\end{tikzcd}
	\end{equation}
	relating the "commutative" graph complex $\GC_0$ to the ribbon graph complex.
	All arrows respect a natural grading, which is given on the ribon graphs by the genus, and on the "ordinary" graphs by the loop order.
	The arrows labelled "$\simeq_{\geq 2}$" are quasi-isomorphisms in loop orders (resp. genus) $\geq 2$. 
	
	\item It is easy to check that the map \eqref{equ:LieBmapdef} in fact factors through the involutive Lie bialgebra properad $\ILieB$, and one can hence consider the deformation complexes 
	\begin{equation*}
	\Def(\La\ILieB_\infty \xrightarrow{} \RGra_0).
	\end{equation*}
	In this way one can obtain further algebraic structures on the Kontsevich-Penner ribbon graph complex, see \cite[section 4.3]{MW}, but this is not important for us in the present paper.
\end{enumerate}

\subsection{Ribbon graph complex with labelled punctures}
The Merkulov-Willwacher construction recalled in the previous subsection a priori only considers the version of the ribbon graph complex for the moduli space with "antisymmetrized" marked points.   
Here we now upgrade the properadic definition of the ribbon graph complex to the case of labelled boundary components.
Concretely, we will copy the version of the definition of the oriented hairy graph complex $\HHOGCS_1$ via properadic deformation complexes as in section \ref{sec:HHOGCSprop}.
To this end, we define the map of properads 
\[
* : \Omega(\AC^*)\to \RGra_0
\]
as the composition of properad maps (see section \ref{sec:HHOGCSprop} for the first map)
\[
\Omega(\AC^*) \to \La\LieB_\infty \to \La\LieB \to \RGra_0.
\]
We then consider the properadic deformation complex of that map, which reads
\begin{align*}
\Def\left(\Omega(\AC^*) \xrightarrow{*} \RGra_0\right)
&= 
\prod_{r,s} \Hom_{\bbS_r\times \bbS_s}\left(\AC^*(r,s), \RGra_0(r,s)\right)
\\\cong
\prod_{r,s} \RGra_0(r,s)^{\bbS_s}
,
\end{align*}
disregarding the differential for now. Elements of this complex can be understood as formal series of ribbon graphs with unidentifiable vertices, and boundary components labelled by numbers $1,\dots,r$. The differential $\delta$ can be seen to be acting by vertex splitting. There are again two gradings that are preserved by the differential, namely the grading by the number $r$ of boundary components and the genus of the ribbon graph. (I.e., the genus of the surface obtained by slightly thickening the graph.)
We will denote the subcomplex of genus $g$ and with $r$ components by 
\[
B_g\Def\left(\Omega(\AC^*) \xrightarrow{*} \RGra_0\right)^r	
\subset \Def\left(\Omega(\AC^*) \xrightarrow{*} \RGra_0\right).
\]
Then the version of the Kontsevich-Penner ribbon graph complex for genus $g$ surface with $r$ labelled points (with $2g+r\geq 3$) is the subcomplex
\[
B_g \RGC^S \subset B_g\Def\left(\Omega(\AC^*) \xrightarrow{*} \RGra_0\right)^r		
\]
consisting of the ribbon graphs all of whose vertices are at least trivalent. More concretely, one has that 
\[
H(B_g\RGC^{[r]})\cong H_c(\MM_{g,r})\otimes \Q[-1]^{\otimes r}.	
\]
Furthermore, we will again consider the pieces of various genera together and define the subcomplex
\[
\RGC^{[r]} := \prod_{g} B_g\RGC^{[r]}\subset \Def\left(\Omega(\AC^*) \xrightarrow{*} \RGra_0\right)^r.
\]
Finally, we shall take the notational liberty to label the $r$ boundary components by the elements of any finite set $S$ with $r$ elements instead, and use the the notation
\[
\RGC^{S} \cong \RGC^{[r]},
\]
for the resulting complex,
where we have fixed some bijection $S\cong [r]$.

\subsection{An extension of the map of Merkulov-Willwacher}
The properadic definition of the ribbon graph complex with labelled punctures makes it easy to extend the maps of diagram \eqref{equ:Def RGC GC} to this case. More concretely, just by functoriality of the properadic deformation complexes we have the map of complexes
\[
\Def(\Omega(\AC^*) \xrightarrow{*} \La\LieB_\infty)
\to
\Def(\Omega(\AC^*)\xrightarrow{*} \RGra_0)
\]
induced by the composition with the map \eqref{equ:LieBmapdef}.
This map induces maps on the subcomplexes we have considered above:
\[
\begin{tikzcd}
	& \Def(\Omega(\AC^*) \xrightarrow{*} \La\LieB_\infty)  \ar{r}
	& \Def(\Omega(\AC^*)\xrightarrow{*} \RGra_0)
\\
\HOGC^{[r]}_1 \ar[hookrightarrow]{r} & \HHOGC^{[r]}_1 \ar{r}{\text{induced}} \ar[hookrightarrow]{u}
& \RGC^{[r]}\ar[hookrightarrow]{u}
\end{tikzcd}.
\]

By Theorem \ref{thm:main} we furthermore have a quasi-isomorphism
\[
\HOGC^{[r]}_{1} \xrightarrow{\simeq} \HGC^{[r]}_{0}.
\]
Hence we obtain in particular a map 
\begin{equation}\label{equ:hairyMW}
H(\HGCS_{0}) \to H(\RGC^{S}).
\end{equation}
It is natural to raise the following conjecture.
\begin{conjecture}
	After identifying the genus $g$ part of $H(\RGC_{S,0})$ with $H_c(\MM_{g,S})[-|S|]$ the map \eqref{equ:hairyMW} agrees with the one obtained by Chan-Galatius-Payne \cite{CGP2} (see Theorem \ref{thm:CGP2}), possibly up to an overall conventional multiplicative constant.
\end{conjecture}

\subsubsection{Simplest nontrivial example (genus $1$ case)}
To illustrate that the map \eqref{equ:hairyMW} is fairly explicit, let us work it out in the genus 1 situation.
The nontrivial cohomology classes in $H(\HGC^{[r]}_{0})$ are represented by linear combinations of "loop" graphs of the form 
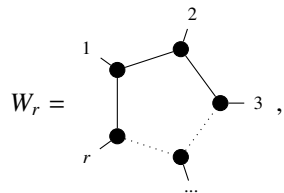
\[
W_r =
\begin{tikzpicture}[scale=.5]
\node [int] (v4) at (1.5,0) {};
\node [int] (v5) at (0.46,1.43) {};
\node [int] (v2) at (-1.21,0.88) {};
\node [int] (v1) at (-1.21,-0.88) {};
\node [int] (v3) at (0.46,-1.43) {};
\draw  (v2) edge (v1);
\draw[dotted]  (v1) edge (v3);
\draw[dotted]  (v3) edge (v4);
\draw  (v4) edge (v5);
\draw  (v5) edge (v2);
\node  (v6) at (-2.02,1.47) {$\scriptstyle 1$};
\node  (v7) at (0.77,2.38) {$\scriptstyle 2$};
\node  (v8) at (2.5,0) {$\scriptstyle 3$};
\node  (v9) at (0.77,-2.38) {$\scriptstyle \cdots$};
\node  (v10) at (-2.02,-1.47) {$\scriptstyle r$};
\draw  (v2) edge (v6);
\draw  (v5) edge (v7);
\draw  (v4) edge (v8);
\draw  (v3) edge (v9);
\draw  (v1) edge (v10);
\end{tikzpicture}\, ,
\] 
or graphs obtained from $W_r$ by permuting the $r$ labels.
Hence is suffices to compute the image of $W_r$ in $H(\RGC^{[r]}_{0})$ for our purposes.
First it is easy to check that under the explicit map of Theorem \ref{thm:main} the graph $W$ is in the image of 
\[
W_r' =
\begin{tikzpicture}
\node[int] (v3) at (0:1) {};
\node[int] (v2) at (72:1) {};
\node[int] (v1) at (144:1) {};
\node[int] (vr) at (-144:1) {};
\node[int] (vd) at (-72:1) {};
\node[int] (v3i) at (36:1) {};
\node[int] (v2i) at (108:1) {};
\node[int] (v1i) at (180:1) {};
\node[] (vri) at (-108:1) {$\scriptstyle \cdots$};
\node[] (vdi) at (-36:1) {$\scriptstyle \cdots$};
\node (v3h) at (0:2) {$\scriptstyle 3$};
\node (v2h) at (72:2) {$\scriptstyle 2$};
\node (v1h) at (144:2) {$\scriptstyle 1$};
\node (vrh) at (-144:2) {$\scriptstyle r$};
\node (vdh) at (-72:2) {$\scriptstyle \cdots$};
\draw[-latex] (v3i) edge (v3) edge (v2)
(v2i)  edge (v1) edge (v2)
(v1i) edge (v1) edge (vr)
(vri) edge (vr) edge (vd)
(vdi) edge (vd) edge (v3)
(v3) edge (v3h) (v1) edge (v1h) (v2) edge (v2h) (vr) edge (vrh) (vd) edge (vdh)
;
\end{tikzpicture}
\, \in \HOGC^{[r]}_{1}.
\]
Mapping this to $\HHOGC^{[r]}_{1}\subset \Def(\LieB_\infty\to \RGra_0)$ we hence obtain 
\[
W_r'' = 
\begin{tikzpicture}
\node[int] (v3) at (0:1) {};
\node[int] (v2) at (72:1) {};
\node[int] (v1) at (144:1) {};
\node[int] (vr) at (-144:1) {};
\node[int] (vd) at (-72:1) {};
\node[int] (v3i) at (36:1) {};
\node[int] (v2i) at (108:1) {};
\node[int] (v1i) at (180:1) {};
\node[] (vri) at (-108:1) {$\scriptstyle \cdots$};
\node[] (vdi) at (-36:1) {$\scriptstyle \cdots$};
\node (v3h) at (0:2) {$\scriptstyle 3$};
\node (v2h) at (72:2) {$\scriptstyle 2$};
\node (v1h) at (144:2) {$\scriptstyle 1$};
\node (vrh) at (-144:2) {$\scriptstyle r$};
\node (vdh) at (-72:2) {$\scriptstyle \cdots$};
\draw[-latex] (v3i) edge (v3) edge (v2)
(v2i)  edge (v1) edge (v2)
(v1i) edge (v1) edge (vr)
(vri) edge (vr) edge (vd)
(vdi) edge (vd) edge (v3)
(v3) edge (v3h) (v1) edge (v1h) (v2) edge (v2h) (vr) edge (vrh) (vd) edge (vdh)
;
\node[] (v3ii) at (36:.2) {};
\node[] (v2ii) at (108:.2) {};
\node[] (v1ii) at (180:.2) {};
\node[] (vrii) at (-108:.2) {};
\node[] (vdii) at (-36:.2) {};
\draw[-latex] (v3ii) edge (v3i) (v1ii) edge (v1i) (v2ii) edge (v2i) (vrii) edge (vri) (vdii) edge (vdi);
\end{tikzpicture}
 +(\cdots),
\]
where $(\cdots)$ is a linear combination of graphs that contain vertices of valence $\geq 4$. Under the map to $\RGC^{[r]}_{0}$ those graphs are sent to zero by construction, hence only the leading term of $W_r''$ is relevant. To compute its image ribbon graph(s) we have to replace the trivalent vertices by ribbon "pairs of pants" as in \eqref{equ:LieBmapdef}, and then apply the properadic compositions in $\RGra$.
One quickly checks that this yields the ribbon graph
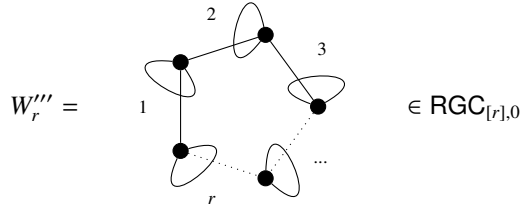
\[
W_r''' =
\begin{tikzpicture}[every loop/.style={draw, -, looseness=15}]
\node[int] (v3) at (0:1) {};
\node[int] (v2) at (72:1) {};
\node[int] (v1) at (144:1) {};
\node[int] (vr) at (-144:1) {};
\node[int] (vd) at (-72:1) {};
\draw (v2) edge (v3) edge (v1)
(vr)  edge (v1) edge[dotted] (vd)
(v3) edge[dotted] (vd);
\draw (v3) edge [in=30,out=160, loop] (v3);
\draw (v1) edge [in=174,out=304, loop] (v1);
\draw (v2) edge [in=102,out=232, loop] (v2);
\draw (vr) edge [in=246,out=376, loop] (vr);
\draw (vd) edge [in=318,out=448, loop] (vd);
\node at (-180:1.3) {$\scriptstyle 1$};
\node at (108:1.3) {$\scriptstyle 2$};
\node at (36:1.3) {$\scriptstyle 3$};
\node at (-108:1.3) {$\scriptstyle r$};
\node at (-36:1.3) {$\scriptstyle \cdots$};
\end{tikzpicture}
  \in \RGC_{[r],0}
\] 
as our final result.

\section{On the work of Chan-Galatius-Payne}\label{sec:CGP}
The goal of this section is to describe an independent, shortened proof of Theorems \ref{thm:CGP2} and \ref{thm:CGP1}, and connect those Theorems to the results of the previous sections.

\subsection{Getzler-Kapranov graph complex }
Let us recall some operadic facts about the moduli spaces of curves understood by Getzler-Kapranov \cite{GK}. First they note that the collection of the Deligne-Mumford compactified moduli spaces of stable curves $\bMM = \left\{ \bMM_{g,n} \right\}_{g,n}$ forms a modular operad. 
Hence the same is true for the corresponding chains operad. 
Similarly, they show that a version of the differential forms on the open moduli spaces $\MM = \left\{ \MM_{g,n} \right\}_{g,n}$ can be made into a (topological) modular cooperad, up to certain degree shifts. Finally, they show (see \cite[Proposition 6.11]{GK}) that both modular (co)operads are "Koszul dual" to each other, in the sense that they are related via the Feynman transform.
Furthermore, it was shown in \cite{GNPR} that the modular operad $\bMM$ is formal so that one can replace its (co)chains (co)operad by the (co)homology.
Combining both results hence motivates the following definition.
\begin{defi}
	The Getzler-Kapranov graph complex
	$$
	\GK=\FT \left(H_\bullet\left(\bMM\right)\right)
	$$
	is the Feynman transform of the homology operad of the Deligne-Mumford compactified moduli spaces of stable curves.
\end{defi}
Let us spell out the definition more explicitly.
The ($\kk$-)modular operad $\GK$ is a collection of dg vector spaces $\GK(g,n)$ of genus $g$ operations with $n$ inputs. We will sometimes also index our inputs with some set $S$ (so that $n=|S|$) and write $\GK(g,S)$ accordingly. 
The elements of $\GK(g,S)$ are series of (isomorphism classes of) $H^\bullet(\bMM)$-decorated genus $g$ oriented graphs. Explicitly, these are triples $\Gamma=(\gamma,\DD, \ori)$ as follows:
\begin{itemize}
	\item $\gamma$ is an undirected graph with external legs (or hairs) indexed by $S$.
	\item Each vertex $x$ of $\gamma$ is decorated by an element of $H^\bullet\left(\bMM_{g_x,\gstar(x)}\right)$, with $g_x$ a non-negative integer associated to $x$, and $\gstar(x)$ being the set of half-edges incident at $x$.
	We may collect all these decorations into one element
	\[
	\DD \in \otimes_G H^\bullet\left(\bMM\right)
	\]
	of the graph-wise tensor product of $H^\bullet\left(\bMM\right)$.
	\item The orientation $\ori$ is an ordering of the set of edges of $\Gamma$.
	\item We identify two such triples up to sign if they can be transformed into each other by applying an isomorphism of graphs, and by changing the orientations. The sign is obtained from the one on decorations and the permutation in the orientation change in the natural way.
	\item The cohomological degree is 
	\begin{equation}\label{equ:GKdegree}
	  |\DD| + e(\gamma) 
	\end{equation}
	with $e(\gamma)$ being the number of edges of $\gamma$.
	\item The genus of the graph is $g=\sum_x g_x+(\text{\#loops})$,
	\item We require that for each vertex $x$ we have 
	\[
	2g_x + |\gstar(x)| \geq 3.
	\]
\end{itemize}
The differential on $\Gamma$ is given by splitting vertices and producing one tadpole
\[
\delta \Gamma = \sum_{x\in V(\Gamma)} split_x(\Gamma)+ tadpole_x(\Gamma),
\]
with $split_x(\Gamma)$ being obtained by replacing $x$ by 2 vertices and applying the cooperadic cocomposition to the decoration at $x$, see \cite{GK}, and $tadpole_x(\Gamma)$ is obtained by using the modular cocomposition to produce a tadpole at $x$. (This latter operation reduces $g_x$ by one.)
It is clear from \eqref{equ:GKdegree} that this operation has cohomological degree $+1$.

\begin{defi}
	We define the weight grading on $\GK$ to be the grading by the total degree on the decorations $\DD$, i.e., 
	\[
	w(\Gamma) := |\DD|.
	\]
\end{defi}
It is clear that the weight grading is untouched by the differential and hence indeed defines a grading on $\GK$.

As outlined in the beginning of this subsection one may then extract the following result from the literature.
\begin{thm}[\cite{GK},\cite{GNPR}]\label{thm:GKHMM}
	The Getzler-Kapranov graph complex computes the compactly supported cohomology of the open moduli spaces,
	\[
	H^\bullet(\GK) \cong H^{\bullet}_c(\MM).
	\]
\end{thm}
In fact, one should rather consider the right-hand side as the weight associated graded of the compactly supported cohomology. On the level of (degree-)graded vector spaces that we consider here this is irrelevant, however.
\begin{proof}
We recall the following notation and results from Getzler-Kapranov \cite{GK}.
First they describe a modular operad in nuclear Fr\'echet (NF) spaces $\CC_\bullet(\bMM, D)$.
They also describe a $\kk$-modular operad in nuclear DF-spaces $\pGrav$ modeling the (degree shifted) chains on $\MM$.
More concretely, 
\[
H_d(\pGrav(g,r)) = H_{d-6(g-1)-2r}(\MM_{g,r}).
\]
They furthermore extend the Feynman transform $\FT$ to a topological version $\FT^{top}$ defined on NF and nuclear DF spaces by replacing tensor products by their projectively completed versions and dually by strong duals. (This presents no major problem, since the Feynman transform only involves finite direct sums and finite tensor products of duals of the argument.)

Our starting point is then the result \cite[Proposition 6.11]{GK} that one has an isomorphism
\[
\FT^{top}_{\kk}\pGrav \cong \CC(\bMM, D),
\]
Next, we know from \cite{GNPR} that $H_\bullet(\bMM)$ has a minimal model (of finite type) $M$, and furthermore that $\bMM$ is a formal modular operad. This allows us to extend the above isomorphism to a zigzag of (quasi-)isomorphisms
\[
\FT^{top}_{\kk}\pGrav \cong \CC(\bMM, D)
\leftarrow M \to H_\bullet(\bMM).
\]
All objects can be considered as modular operads in NF spaces and the maps are continuous -- mind that the two objects on the right are of finite type.
Applying the topological Feynman transform again allows us to write the following zigzag of quasi-isomorphisms
\[
\pGrav \xleftarrow{\sim} \FT^{top}\FT^{top}_{\kk}\pGrav \cong \FT^{top} \CC(\bMM, D)
\leftarrow \FT^{top}M = \FT \to \FT H_\bullet(\bMM).
\]
This shows our result. Here we used the fact (see \cite[section 5]{GK}) that the Feynman transform is a homotopy functor and that $\FT^{(top)}$ is the homotopy inverse to $\FT^{(top)}_{\kk}$.
\end{proof}

\subsection{An independent proof of Theorems \ref{thm:CGP2} and \ref{thm:CGP1} }
We may show Theorems \ref{thm:CGP1} and \ref{thm:CGP2} together.
We use the complex $\GK$ to compute $H(\MM)$ via Theorem \ref{thm:GKHMM}
First it is clear that the weight 0 part is a direct summand of $\GK$.
But since $H_0(\bMM)=\Com^{mod}$ the weight zero part is given by the commutative graph complexes $\HGC^{S,mod}_0$ of section \ref{sec:undirected_complexes_modular}.
This is in turn quasi-isomorphic (in the stable situation) to the graph complex $\HGC^S_0$ by Lemma \ref{lem:HGCmod}.
\hfill\qed

Furthermore, we note that via the modular operad maps 
\[
\Com^{mod} \to H_\bullet(\bMM) \to \Com
\]
one has maps of complexes 
\begin{align*}
B_g\HGC^S_0 \to W_0\GK(g,S) \to B_g\HGC^{S,mod}_{0} \to B_g \HGC^S_{0} 
\end{align*}
see also section \ref{sec:undirected_complexes_modular}.


\subsection{Antisymmetrized Getzler-Kapranov complex and extra differential}
We may mimic the constructions of the hairy graph complexes $\HGCS_{n}$ and $\HGC_{m,n}$ of section \ref{sec:undirected_complexes} and define the following objects:
\begin{align*}
\HGK^S &:= \prod_{g\geq 1}  \GK(g,S) \otimes \sgn_S[-|S|]\\
\HGK_0 &:= \prod_{r\geq 1} (\HGK_r)^{S_r}.  
\end{align*}
As before we will denote the genus $g$ pieces by $B_g \HGK^S=\GK(g,S) \otimes \sgn_S[-|S|]$ and $B_g\HGK_0$.
Obviously, the cohomology of $\HGK_0$ may be identified with the compactly supported cohomology of the moduli spaces $\MM_{g,r}$ with antisymmetrized punctures, up to conventional degree shifts. We call the number $r$ of marked points also the number of hairs to unify the notation with the other graph complexes considered.

Now, forgetting one marked point on the compactified moduli spaces induces a natural map 
\[
\pi^*:H^\bullet(\bMM_{g,n}) \to H^\bullet(\bMM_{g,n+1}).
\]
We may apply this to every vertex to obtain a degree $+1$-map 
\begin{align*}
\Delta : \HGK_0 \to \HGK_0 \\
\Gamma \mapsto \sum_{x\in V(\Gamma)} (addhair)_x(\Gamma)
\end{align*}
where $(addhair)_x(\Gamma)$ applies $\pi^*$ to the decoration at vertex $x$.
To fix the signs, we declare that the newly added hair becomes the first in the ordering.
This increases the number of hairs $r$ by one.

\begin{lemma}
The operation $\Delta$ is compatible with the splitting differential $\delta$, i.e., $(\delta + \Delta)^2=0$.
\end{lemma}
\begin{proof}[Proof sketch]
One can check that $\Delta^2=0$ and $\Delta\delta+\delta\Delta$ vanish separately. The first equation is immediate since we antisymmetrized over the markings.
The second boils down to the commutativity of the diagrams 
\[
\begin{tikzcd}
\bMM_{g_1,n_1+1} \times \bMM_{g_2,n_2+1} \ar{r}\ar{d}& \bMM_{g_1+g_2,n_1+n_2} \ar{d}\\
\bMM_{g_1,n_1} \times \bMM_{g_2,n_2+1} \ar{r} & \bMM_{g_1+g_2,n_1+n_2-1}
\end{tikzcd}
\]
and
\[
\begin{tikzcd}
\bMM_{g,n+1} \ar{r}\ar{d}& \bMM_{g+1,n-1} \ar{d}\\
\bMM_{g,n} \ar{r}& \bMM_{g+1,n-2}
\end{tikzcd}\, ,
\]
where the horizontal arrows are modular operadic compositions (appearing in $\delta$) and the vertical arrows are forgetful maps, forgetting one puncture (appearing in $\Delta$).
\end{proof}

It is also clear that the inclusion 
\[
\HGC_{0,0} \to \HGK_0
\]
intertwines the operation $\Delta:=[m_0,-]$ on $\HGC_{0,0}$ of the introduction and the operation $\Delta$ on $\GK_0$ just defined.

We also note that $\Delta$ is defined in the same way on $\HGK^\emptyset$, giving rise to a map of complexes
\[
\Delta : \HGK^\emptyset[-1] \to \HGK_0.
\]
and this is true for both the differentials $\delta$ and $\delta+\Delta$ with which we may equip the right-hand side.

Altogether we obtain the following commutative diagram of complexes
\begin{equation}\label{equ:HGCHGK cd}
\begin{tikzcd}
(\HGC_{0,0}, \delta+\Delta) \ar{r} & (\HGK_0, \delta+\Delta) \\
(\GC_0[-1] ,\delta) \ar{r} \ar{u}{\Delta} & \HGK^\emptyset[-1],\delta)\ar{u}{\Delta}
\end{tikzcd}\, .
\end{equation}
The left-hand colum is the weight 0 part of the right-hand column. Furthermore, note that operation $\Delta$ preserves the genus (in the sense of the loop order on the left), and so do the horizontal maps.
We furthermore know that the left-hand vertical arrow induces an isomorphism on cohomology in genera $\geq 2$.
This is also true for the right-hand vertical arrow:

\begin{thm}\label{thm:Delta cohomology}
The right-hand vertical arrow of \eqref{equ:HGCHGK cd} induces an isomorphism on cohomology in genera $g\geq 2$.
In particular, $(\HGK_0,\delta+\Delta)$ computes the compactly supported cohomology of the moduli space $\MM_g$ in these genera, up to a degree shift by one.
\end{thm}
\begin{proof}
	We consider the mapping cone $C$ of the map in question, i.e., 
	\[
	C=(\HGK^\emptyset\oplus \HGK_0, \delta+\Delta).
	\]
	This is a natural extension of $\HGK_0$ is that one merely allows graphs without hairs.
	
	Now we filter the complex $C$ by the number of internal (non-hair) edges in graphs. This is a bounded above complete descending filtration. Hence we are done if we can show that the associated graded complex is acyclic in positive genera.
	Furthermore, $\delta$ creates one internal edge, and $\Delta$ none, hence we may take $(C,\Delta)$ as the associated graded.
	Now $(C,\Delta)$ is a product of direct summands of tensor products of complexes associated to single vertices.
	The complex associated to a single vertex with $k$ internal legs has the form
	\[
	C_{g,k} = (\prod_{r\geq r_0} (H(\bMM_{g,k+r})\otimes \Q[-1]^{\otimes r})_{\bbS_r}, \Delta),
	\]
	with $\Delta$ being the pullback for the forgetful map forgetting one of the "antisymmetric" marked points as above. The lower bound for the product is $r_0=\max(3-2g-k,0)$ and comes from the stability condition. Now if the genus of the graph, i.e., the number of loops plus the total number of genera of vertices, is at least 2, we can always find a vertex such that $2g+k\geq 3$.
	Hence it suffices to show that $H(C_{g,k},\Delta)=0$ in this case.
	We will do this by constructing an explicit homotopy for $\Delta$.
	To this end let 
	\[
	\pi_j : \bMM_{g,k+r+1} \to \bMM_{g,k}
	\]
	be the forgetful map forgetting the location of the $j$-th of the $k$ "antisymmetric" marked points. Then 
	\[
	\Delta = \sum_{j=1}^{r+1} (-1)^{j-1} \pi_j^*.
	\]
	Now let $\Psi_j\in \bMM_{g,k+r+1}$ be the $\Psi$-class at the $j$-th such marking, abusively hiding the number $r$ from the notation. We can assume that it is normalized such that
	\[
	(\pi_j)_{!} \Psi_j = 1.
	\] 
	Then we define our homotopy $h: C_{g,k}  \to C_{g,k}[-1]$ on $\alpha \in H(\bMM_{g,k+r})$ as
	\[
	h(\alpha) = \sum_{j=1}^r (-1)^{j-1} (\pi_j)_{!} (\Psi_j \wedge \alpha)
	\]
	if $r\geq 1$ and set $h(\alpha)=0$ if $r=0$.
	One computes (say first for $r\geq 1$)
	\begin{align*}
	h(\Delta(\alpha)) &=
	\sum_{i=1}^{r+1}\sum_{j=1}^{r+1} (-1)^{i+j} (\pi_i)_{!} (\Psi_i \wedge  \pi_j^*\alpha)
	\\&=
	\sum_{1\leq i<j\leq r+1}(-1)^{i+j}
	(-1)^{i+j} \pi_{j-1}^* (\pi_i)_{!} (\Psi_i \wedge  \alpha)
	+
	\sum_{i=1}^{r+1}
	\underbrace{(\pi_i)_{!} \Psi_i)}_{=1}\alpha
	+
	\sum_{1\leq j<i\leq r+1}
	(-1)^{i+j} \pi_j^* (\pi_{i-1})_{!} (\Psi_i \wedge \alpha)
	\\&=
	(r+1)\alpha - \Delta (h(\alpha)).
	\end{align*}
	Hence we see that $H(C_{g,k})=0$ as desired.
	
\end{proof}

Note that the lower right-hand complex $\HGK^\emptyset$ in \eqref{equ:HGCHGK cd} computes the compactly supported cohomology of the moduli space of (non-pointed) curves by Theorem \ref{thm:GKHMM}. The cohomology of the upper right-hand complex has a spectral sequence (from the filtration on the number of marked points) whose first page is 
\[
\prod_{g\geq 0,n\geq 1 \atop 2g+n\geq 3} 
\left(H_c(\MM_{g,n})\otimes \Q[-1]^{\otimes n}\right)_{\bbS_n} \Rightarrow H(\HGK_0,\delta+\Delta).
\] 
Hence the Theorem gives a relation between the cohomology of the moduli spaces of marked and non-marked curves.
One may also give this a conjectural geometric interpretation.

\begin{conjecture}\label{conj:Delta interpretation}
	Under the identification of $\HGK_0$ with (some version of) compactly differential forms on the moduli spaces of curves with antisymmetrized points, the operation $\Delta$ corresponds geometrically to the pullback under the forgetful map $\pi: \MM_{g,n+1} \to \MM_{g,n}$, forgetting one marked point, and the latter pullback induces a well defined operation on the compactly supported differential forms.
\end{conjecture}

\begin{rem}
The main problem in showing Conjecture \ref{conj:Delta interpretation} in our framework is that we use the formality result of \cite{GNPR}. This covers only the modular operad structure, but not the forgetful maps, forgetting some of the marked points.
 

We finally remark that Alexey Kalugin probably has a proof of the above conjecture (personal communication).
\end{rem}

\begin{rem}
	By similar arguments as in the proof of Theorem \ref{thm:Delta cohomology} we may in fact compute the part of the cohomology of $(\HGK_0,\delta+\Delta)$ in genus 0 and 1 as well.
	Concretely, following the argument and using the same notation as in that proof, we see that we need to consider only graphs all of whose vertices $x$ satisfy $2g+k\leq 2$.
	The cohomology of the complexes $(C_{g,k}, \Delta)$ is computed as follows:
	\begin{itemize}
		\item In the cases $g=1$, $k=0$ and $g=0$, $k=2$ the cohomology $H(C_{g,k}, \Delta)$ is one-dimensional, corresponding to a single hair attached to the vertex.
		\item In the remaining cases $g=0$, $k<2$ the complex $C_{g,k}$ is 0, because there need to at least two hairs (markings) to satisfy stability, but then the corresponding classes are killed by the anti-symmetrization.
	\end{itemize}

	Overall, one sees that one has non-trivial cohomology only in the genus 1 case, and there the remaining diagrams are the ``hedgehog'' graphs of the form
	
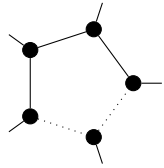
\[
		\begin{tikzpicture}[scale=.5]
			\node [int] (v4) at (1.5,0) {};
			\node [int] (v5) at (0.46,1.43) {};
			\node [int] (v2) at (-1.21,0.88) {};
			\node [int] (v1) at (-1.21,-0.88) {};
			\node [int] (v3) at (0.46,-1.43) {};
			\draw  (v2) edge (v1);
			\draw[dotted]  (v1) edge (v3);
			\draw[dotted]  (v3) edge (v4);
			\draw  (v4) edge (v5);
			\draw  (v5) edge (v2);
			\node  (v6) at (-2.02,1.47) {};
			\node  (v7) at (0.77,2.38) {};
			\node  (v8) at (2.5,0) {};
			\node  (v9) at (0.77,-2.38) {};
			\node  (v10) at (-2.02,-1.47) {};
			\draw  (v2) edge (v6);
			\draw  (v5) edge (v7);
			\draw  (v4) edge (v8);
			\draw  (v3) edge (v9);
			\draw  (v1) edge (v10);
			\end{tikzpicture}	
	\] 
	that live in weight 0, and span the compactly supported cohomology of the $\MM_{1,n}$ with antisymmetrized markings.
\end{rem}


\subsection{Incorporating ribbon graphs, and the (conjectural) big picture}\label{sec:big picture}

We can now put together the maps of section \ref{sec:HOGCtoRGC} and those of the previous subsection to obtain a big commutative diagram of complexes (straight arrows)
\[
\begin{tikzcd}
\HGK^\emptyset[-1] \ar{r}{\simeq_{\geq 2}} & (\HGK_0,\delta+\Delta) \ar[-, dashed]{ddr}{\simeq?}  & \\
\GC_0[-1] \ar{r}{\simeq_{\geq 2}} \ar{u} & (\HGC_{0,0},\delta+[m_0,-]) \ar{u} &  \\
\OGC_1[-1] \ar{r}{\simeq_{\geq 2}} \ar{u}{\simeq_{\geq 2}} & (\HOGC_{0,1},\delta+[m_1,-]) \ar{r}  \ar{u}{\simeq} &  (\RGC_0,\delta+\Delta) 
\end{tikzcd}\, ,
\]
where the symbol $\simeq_{\geq 2}$ shall indicate that the map is a quasi-isomorphism in genera $\geq 2$. Recall in particular the definition of the "Bridgeland" differential $\Delta$ on the ribbon graph complex $\RGC_0$ from above, see \eqref{equ:bridgelanddiff_illustration}. 
The middle row is the weight 0 part of the first row.
Given that 
$$H(\HGK_0,\delta)\cong H(\RGC_0, \delta)\cong \prod_{r\geq 1, g\geq 0 \atop 2g+1\geq 3} (H_c(\MM_{g,r})\otimes \Q[-1]^{\otimes r})_{\bbS_r}$$
it is hence natural to conjecture that there is a quasi-isomorphism (possibly a zigzag) between $(\HGK, \delta+\Delta)$ and $\RGC,\delta+\Delta)$  (the dashed line) that makes the right hand triangle commute.

This would then in particular imply the following conjecture of A. C\v ald\v araru (personal communication)
\begin{conjecture}[C\v ald\v araru] \label{conj:caldararu}
	The cohomology of the ribbon graph complex with altered differential $(\RGC,\delta+\Delta)$ can be naturally identified with the compactly supported cohomology of the moduli spaces of curves without marked points in genera $g\geq 2$.
	\[
	H(\RGC^{g},\delta+\Delta) \cong H_c(\MM_g)[-1].
	\]
\end{conjecture}


\begin{thebibliography}{1}

\bibitem{AssarMarko}
Assar Andersson, Marko \v Zivkovi\'c.
\newblock Hairy graphs to ribbon graphs via a fixed source graph complex.
\newblock preprint arXiv:1912.09438, 2019.


\bibitem{AT}
 G.~{Arone} and V.~{Turchin}.
 \newblock {Graph-complexes computing the rational homotopy of high dimensional
   analogues of spaces of long knots}.
 \newblock {\em Ann. Inst. Fourier} 65(1):1--62, 2015.
 





\bibitem{CGP1}
Melody Chan, Soren Galatius and Sam Payne.
\newblock Tropical curves, graph homology, and top weight cohomology of $M_g$,
\newblock arxiv:1805.10186, 2018.

\bibitem{CGP2}
Melody Chan, Soren Galatius and Sam Payne.
\newblock Topology of moduli spaces of tropical curves with marked points
\newblock arxiv:1903.07187, 2019.




\bibitem{GK} E. Getzler and M. Kapranov.
\newblock Modular operads.
\newblock Compositio Math. 110 (1998), no. 1, 65--126. 

\bibitem{GNPR}
F. Guill\'en Santos, V. Navarro, P. Pascual, and A. Roig.
\newblock Moduli spaces and formal operads. 
\newblock Duke Math. J. 129 (2005), no. 2, 291--335. 

\bibitem{FTW}
Benoit Fresse, Victor Turchin and Thomas Willwacher.
\newblock The rational homotopy of mapping spaces of $E_n$ operads
\newblock Preprint, arXiv:1703.06123, 2017.



\bibitem{KWZ}
Anton Khoroshkin, Thomas Willwacher and Marko \v Zivkovi\'c.
\newblock Differentials on graph complexes.
\newblock  {\em Adv. Math.} 307 (2017), 1184--1214.

\bibitem{KWZ2}
Anton Khoroshkin, Thomas Willwacher and Marko \v Zivkovi\'c.
\newblock Differentials on graph complexes II: hairy graphs.
\newblock {\em Lett. Math. Phys.} 107 (2017), no. 10, 1781–1797.


\bibitem{K3}
Maxim Kontsevich.
\newblock Formal (non)commutative symplectic geometry. 
\newblock The Gelfand Mathematical Seminars, 1990--1992, 173--187, Birkhäuser Boston, Boston, MA, 1993. 

\bibitem{LambrechtsTurchin}
Pascal Lambrechts and Victor Turchin.
\newblock Homotopy graph-complex for configuration and knot spaces.
\newblock {\em Trans. Amer. Math. Soc.}, 361(1):207--222, 2009.


\bibitem{lodayval}
J.-L. Loday and B. Vallette.
\newblock Algebraic operads.
\newblock {\em Grundlehren Math. Wiss.}, 346, Springer, Heidelberg, 2012. 

\bibitem{MV-properad} S. Merkulov and B. Vallette, 
\newblock  Deformation theory of representations of prop(erad)s,
\newblock  J. Reine Angew. Math. {\bf 634} (2009), 51--106. and J. Reine Angew. Math. {\bf 636} (2009), 123--174,
arXiv:0707.0889.

\bibitem{MW}
Sergei Merkulov and Thomas Willwacher.
\newblock Props of ribbon graphs, involutive Lie bialgebras and moduli spaces of curves.
\newblock arxiv:1511.07808, 2015.

\bibitem{MWDef}
Sergei Merkulov and Thomas Willwacher.
\newblock Deformation Theory of Lie Bialgebra Properads.
\newblock in \emph{Geometry and Physics: Volume I: A Festschrift in honour of Nigel Hitchin}, 2018, DOI: DOI:10.1093/oso/9780198802013.003.0010.


\bibitem{penner}
R. C. Penner.
\newblock Perturbative series and the moduli space of Riemann surfaces.
\newblock 
J. Differential Geom. 27 (1988), no. 1, 35--53.


\bibitem{ST}
Paul Arnaud Songhafouo Tsopm\'en\'e and Victor Turchin
\newblock Hodge decomposition in the rational homology and homotopy of high dimensional string links. 
\newblock arXiv:1504.00896 (2015).

\bibitem{ST2}
Paul Arnaud Songhafouo Tsopm\'en\'e and Victor Turchin
\newblock Euler characteristics for the Hodge splitting in the rational homology and homotopy of high dimensional string links.
\newblock arXiv:1609.00778 (2016).

\bibitem{Tur1}
Victor Turchin.
\newblock Hodge-type decomposition in the homology of long knots.
\newblock {\em J. Topol.}, 3(3):487--534, 2010.

\bibitem{TW}
Victor Turchin and Thomas Willwacher.
\newblock Relative (non-)formality of the little cubes operads and the algebraic Cerf lemma.
\newblock arXiv:1409.0163, 2014.

\bibitem{TW2}
Victor Turchin and Thomas Willwacher.
\newblock Commutative hairy graphs and representations of $\mathit{Out}(F_r)$.
\newblock arxiv:1603.08855, to appear in J. Top.



\bibitem{grt}
Thomas Willwacher.
\newblock {M. Kontsevich's graph complex and the Grothendieck-Teichm\"uller Lie
  algebra}.
\newblock \emph{Invent. Math.}, 200(3): 671--760 (2015).

\bibitem{Woriented}
Thomas Willwacher 
\newblock The Oriented Graph Complexes.
\newblock Communications in Mathematical Physics volume 334, pages 1649--1666(2015)


\bibitem{MarkoThesis}
Marko \v Zivkovi\'c.
\newblock Graph complexes and their cohomology.
\newblock {\em Doctoral Thesis, University of Zurich}, 2016.

\bibitem{MarkoPaper}
Marko \v Zivkovi\'c.
\newblock Differentials on Graph Complexes III - Deleting a Vertex.
\newblock Lett. Math. Phys. 109 (2019), no. 4, 975–1054.

\bibitem{MultiOriented}
\v Zivkovi\' c, M.
\newblock Multi-directed graph complexes and quasi-isomorphisms between them I: oriented graphs.
\newblock High. Struct. 4(1):266–283, 2020.

\bibitem{MultiSourced}
\v Zivkovi\' c, M.
\newblock Multi-directed Graph Complexes and Quasi-isomorphisms Between Them II: Sourced Graphs.
\newblock Int. Math. Res. Not. IMRN (2019), rnz212.

\end{thebibliography}
\end{document}